\documentclass[12pt]{article}
\usepackage[english]{babel}

\usepackage[letterpaper]{geometry}
\usepackage{pdfpages}
\usepackage{geometry} 
\usepackage{amsmath}  
\usepackage{graphicx} 
\usepackage{float}
\usepackage{color}
\usepackage{amsmath}
\usepackage{amssymb}
\usepackage{amsthm}
\usepackage{array}
\usepackage{xcolor}
\usepackage{float}
\usepackage{bm}

\usepackage{amsmath}
\usepackage{graphicx}
\usepackage[colorlinks=true, allcolors=blue]{hyperref}
\usepackage{caption}
\usepackage{subcaption}

\definecolor{keywordcolor}{rgb}{0.7,0.1,0.1}
\definecolor{commentcolor}{rgb}{0.4,0.4,0.4}
\definecolor{symbolcolor}{rgb}{0.0,0.1,0.6}
\definecolor{sortcolor}{rgb}{0.1,0.5,0.1}

\usepackage{lstlean}

\lstdefinestyle{LeanCode}{
  language=lean,
  basicstyle=\ttfamily\fontsize{8pt}{8pt}\selectfont,
  keywordstyle=[1]{\color{keywordcolor}},
  keywordstyle=[2]{\color{sortcolor}},
  commentstyle=\itshape\color{commentcolor},
  stringstyle=\ttfamily,
  breaklines=true,
  breakatwhitespace=false,
  columns=fullflexible,
  keepspaces=true,
  showstringspaces=false,
  tabsize=2,
  extendedchars=false,
  captionpos=b
}

\usepackage{newunicodechar}
\newunicodechar{⦃}{\ensuremath{\{\!\{}}
\newunicodechar{⦄}{\ensuremath{\}\!\}}}
\newunicodechar{⋂}{\ensuremath{\bigcap}}
\newunicodechar{ₗ}{\ensuremath{_{l}}}
\newunicodechar{ᵢ}{\ensuremath{_{i}}}
\newunicodechar{ℝ}{\ensuremath{\mathbb{R}}}
\newunicodechar{ₜ}{\ensuremath{_{t}}}

\newtheorem{theorem}{Theorem}

\theoremstyle{definition}

\theoremstyle{lemma}
\newtheorem{lemma}[theorem]{Lemma}

\theoremstyle{remark}

\theoremstyle{boldremark}

\newtheorem{corollary}[theorem]{Corollary}

\title{Universal Triangle Covering Curve and Polygonal Chain: Escaping Forest and Fitting Worm}

\author{Zhipeng Deng \thanks{I post this paper to celebrate a truly special, memorable, historical, and proud day when two Chinese mathematicians won the Fields medal!}}

\date{}

\begin{document}
\maketitle

\begin{abstract}
In this paper, we present a general formulation to address the problems of covering curves and polygonal chains with triangle, and fitting these curves into triangle. These problems can be formulated as special cases of Bellman’s lost-in-a-forest problem (escaping triangular forest) and Moser’s worm problem (covered by triangle). We model and reformulate the problem by keeping the curve stationary while allowing the triangle to translate and rotate. Subsequently, we derive the functional minimization formulation with support function constraints to solve. We also prove the equivalence and convergence of the formulas. Finally, we employ numerical methods and present results for covering curves with arbitrary triangles of various angles. We also present some corollaries and variant results, including closed curves and closed polygonal chains. 
\end{abstract}

\\

\noindent \textbf{Classification}

Optimization and Control (math.OC)

Metric Geometry (math.MG)

Discrete Mathematics (cs.DM)

Computational Geometry (cs.CG)

49K30 (Optimal Solutions in Calculus of Variations)

49Q10 (Optimization of Shapes Other Than Minimal Surfaces)

52A40 (Inequalities and Extremum Problems)

\tableofcontents

\section{Introduction}

Bellman’s lost-in-a-forest problem and Moser’s worm problem are challenging, problems in geometry \cite{Finch2004} \cite{Norwood1992}, and they still remain open and unsolved \cite{Brass2005} \cite{Croft2012} . The former was introduced by Bellman in 1956, asking for the shortest path that guarantees escape from a forest of known shape and dimensions, given that the starting position and orientation are unknown \cite{Gross1955}. A decade later, in 1966, Leo Moser proposed the worm problem \cite{Gerriets1974}, which seeks the region of minimum area that can accommodate every continuous planar curve of unit length. 

Regarding Bellman’s lost-in-a-forest problem, rigorously proven optimal escape paths exist for only a severely limited number of highly symmetric forest shapes, such as half plane by Isbell \cite{Isbell1957}, and infinite unit strip by Zalgaller \cite{Zalgaller2005}. For triangular forests, progress has remained fragmented. Besicovitch \cite{Besicovitch1965} provided a zigzag construction yielding an optimal escape path for the equilateral triangle. It was proved in 2006 \cite{Coulton2006}. Numerical bounds for isosceles triangles were later investigated by Gibbs \cite{Gibbs2016} and recently by Temerev \cite{Temerev2026}. For fat convex forest shapes, the diameter is the optimal escape path (Theorem 4 in \cite{Finch2004}). Finch's paper also present the escape path results for regular polygons \cite{Finch2004}. However, no generalized algorithmic solution or mathematical proof for arbitrary triangle, let alone arbitrary convex polygons \cite{Gibbs2016-1}, currently exists in the established literature. In our previous work \cite{Deng2024}, we sought to bridge this gap by proposing a general framework for Bellman’s lost-in-a-forest problem. By reformulating the escape path as a Traveling Salesman Problem with Neighborhoods (TSPN), we also established a proof for the general solution \cite{Deng2026}. Furthermore, we applied this approach to the classic unit strip \cite{Deng2026-1} and yielded the same solution as Zalgaller \cite{Zalgaller2005}.

As for Moser’s Worm problem, classified as a universal curve cover problem \cite{Wetzel2003} \cite{Brass2005} , it has been notably discussed by Finch that the worm problem is dual to the forest problem \cite{Finch2004}. Consequently, research into Moser’s worm has heavily relied on constructing and refining planar covers based directly on the escape paths derived from forest problem, such as Zalgaller’s broadworm and Besicovitch’s zigzag. This approach has led to incremental tightening of the bounds with the area currently known to be bounded below by $0.2322$ and above by $0.2604$ \cite{Poole1973} \cite{Adhikari1989} \cite{Norwood1992} \cite{Norwood2003} \cite{Johnson2004} \cite{Wang2006} \cite{Wetzel2013} \cite{Khandhawit2013} \cite{Movshovich2017} \cite{Wetzel2019} following recent advancements \cite{Movshovich2025}
\cite{Wichiramala2026}. However, earlier studies relied almost on classical geometric constructions for certain shape. 

In our previous works, we explored escape paths of one line \cite{Deng2024}, two lines \cite{Deng2026}, and unit strip \cite{Deng2026-1}. Therefore, this paper will explore three lines, i.e., arbitrary triangle. The contribution of this paper lies in applying the general solution method which we previously proposed for escaping from arbitrary triangular forest and fitting a worm into a triangle. We reformulated the problem as a constrained functional minimization problem by using support function. Unlike the TSPN mixed integer programming approach, this formulation allows for direct discretization and subsequent solutions, offering a simpler, more convenient method. We prove the equivalence and convergence of the formulation. We also extend the methods to find polygonal chains escaping paths of arbitrary triangle, and fitting open and closed polygonal chains into triangles. Finally, we employ numerical methods and present results regarding the escape from arbitrary triangles of various angles and the covering of curves with arbitrary triangles. In particular, we provide novel formulas and obtain numerical results for non-isosceles arbitrary triangles, which was not presented in any previous literature.

\section{Derivation and proof of formulas}

\subsection{Proof of concept}

\begin{figure}[H]
    \centering
    \includegraphics[width=1\linewidth]{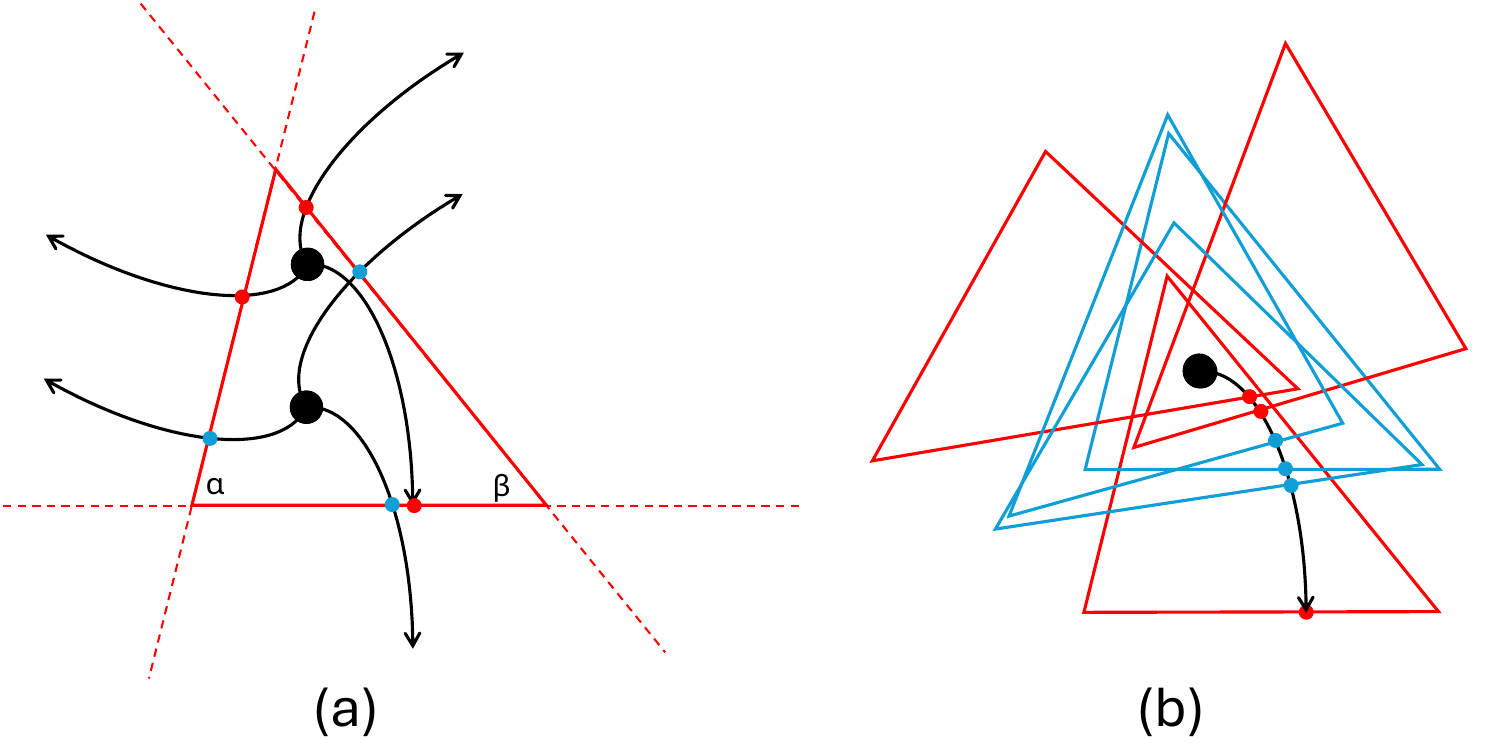}
    \caption{Proof of concept for escaping from arbitrary triangle forest, and fitting worm into triangle}
\end{figure}

Figure 1 shows the proof of concept. We transform the escape path from the triangle across various orientations and locations into a path which is stationary by the translation and rotation of the triangle. We establish that for any translation or rotation of the triangle, the escape path must pass through an escape point on the triangle. We have established and proved this transformation in previous papers \cite{Deng2024} \cite{Deng2026}. In this way, we can convert the discretization lost-in-a-triangular-forest problem as equivalent Traveling Salesperson Problem with Neighborhoods (TSPN). Furthermore, due to the duality between the forest problem and the worm problem—based on Theorem 3 from \cite{Finch2004}, the shortest path to escape a triangular forest corresponds to a worm that can be covered by triangles. Thus, we can also obtain the solution for the worm problem.

Without loss of generality, let us assume a triangle with a base of length 1 on x-axis joining $(0,0)$ and $(1,0)$. Assume $0 < \alpha < \pi, 0 < \beta < \pi, \alpha + \beta < \pi.$ as Figure 1 (a) shows. The third vertex lies in the upper half-plane. 
The equations of the lines forming the three sides are:
\[y=0\]
\[y=x\tan\alpha \]
\[y=(x - 1)\tan(-\beta) \]

The normalized half-space equations are
$
\mathbf{n}_1 = (-\sin\alpha, \cos\alpha), \mathbf{n}_2 = (\sin\beta, \cos\beta), \mathbf{n}_3 = (0, -1),
$
and
$
b_1 = 0, b_2 = \sin\beta, b_3 = 0.
$
Then
$$
\mathcal{T} = \{ \mathbf{x} \in \mathbb{R}^2 : \mathbf{n}_j \cdot \mathbf{x} \le b_j, \ j = 1, 2, 3 \}.
$$
The domain $\mathcal{T}$ is an arbitrary triangle bounded by three lines. 
This is equivalent to ${(x,y)|y \ge 0, y \le x \tan\alpha, y \le (1 - x) \tan\beta}$.

The points within the triangle domain represent the unknown starting positions for the forest problem, with the range of unknown orientations $\theta \in [0, 2\pi)$.

\subsection{Derivation of formula for covering curve}

Based on the proof concept in Fig 1 and Theorem 1 proven in  \cite{Deng2024}
 \cite{Deng2026}, the problem can be reformulated by fixing the escape path starting point (assuming the origin) and considering the intersection of the path with the translating and rotating forest boundaries, which corresponds to the three lines of triangle. Let $\mathbf{s}_i = (\xi_i, \eta_i) \in \mathcal{T}$ be a possible starting point, let $R_\theta$ denote counterclockwise rotation through $\theta$, and write the physical path as

$$ \mathbf{s}_i + R_\theta \mathbf{q}. $$

In path-fixed coordinates, the translated and rotated triangle is

$$ Q(\mathbf{s}_i, \theta) = R_{-\theta}(\mathcal{T} - \mathbf{s}_i) = \{ \mathbf{q} : \mathbf{u}_j(\theta) \cdot \mathbf{q} \le d_j(\mathbf{s}_i), \ j = 1, 2, 3 \} \, , $$

where $ \mathbf{u}_j(\theta) = R_{-\theta}\mathbf{n}_j, d_j(\mathbf{s}_i) = b_j - \mathbf{n}_j \cdot \mathbf{s}_i \ge 0. $

Consequently, its three supporting lines are

\[\mathbf u_j(\theta)\cdot\mathbf q=d_j(\mathbf s_i), j=1,2,3. \]

Explicitly, they are 

\[ x_i\tan \alpha  -y_i+x (\tan \alpha  \cos \theta +\sin \theta )+y (\tan \alpha  \sin \theta -\cos \theta )=0    \]
\[(\tan (-\beta)  \cos \theta +\sin \theta ) x+(\tan (-\beta) \sin \theta-\cos \theta ) y+\left(x_i\tan (-\beta)  -y_i-\tan (-\beta
   )\right)=0     \]
\[  -x\sin \theta  +y\cos \theta  +y_i=0   \]

Thus, based on the TSPN construction from our previous paper \cite{Deng2024}, the general framework for the discretized lost-in-a-triangular-forest problem is follows. Let $\mathcal{A}$ be a discrete uniform grid of points over domain $\mathcal{T} \subset \mathbb{R}^2$. Let the grid spacing approach zero, yielding $M$ interior points. Let $\Theta$ be a discrete set of $N$ angles uniformly distributed in $[0, 2\pi)$. The original discrete optimization problem seeks an ordered sequence of vertices, defined by $\{a_1, a_2, \ldots, a_{MN}\} \in S_{MN}$, a permutation group \cite{Deng2024}, that minimizes the open sequence's total Euclidean variation subject to a product of algebraic intersection constraints for all rotated angles. The product constraint places the selected visit point on at least one supporting line of the transformed triangle. Because the polygonal path starts at the origin, which lies in the transformed triangle, any continuous route reaching such a point must intersect the actual triangle boundary, possibly before reaching that selected point.
\begin{equation}
\begin{aligned}
    \text{minimize}_{a \in S_{MN}} \quad & \sqrt{x_{a_1}^2 + y_{a_1}^2} + \sum_{i=2}^{MN} \sqrt{(x_{a_i} - x_{a_{i-1}})^2 + (y_{a_i} - y_{a_{i-1}})^2}, \\
    \text{s.t.} \quad & [x_i\tan \alpha - y_i + x_{a_h} (\tan \alpha \cos \theta + \sin \theta) + y_{a_h} (\tan \alpha \sin \theta - \cos \theta)] \\
    & [(\tan (-\beta) \cos \theta + \sin \theta) x_{a_h} + (\tan (-\beta) \sin \theta - \cos \theta) y_{a_h} \\
    & \quad + \left(x_i\tan (-\beta) - y_i - \tan (-\beta)\right)] \\
    & [-x_{a_h}\sin \theta + y_{a_h}\cos \theta + y_i] = 0, \\
    & \forall i \in \{1, \dots, M\},\forall k \in \{1, \dots, N\}, \exists h \in \{1, \dots, MN\},
\end{aligned}
\end{equation}
As the grid densities $M, N \to \infty$, the uncertainty grids become dense, the finite scenario constraints approximate the robust requirement over all starting positions and orientations. The bounded-length polygonal apths are compact after arc-length reparameterization, and subsequences converge uniformly to rectifiable continuous paths\cite{Deng2026}. We parameterize the sequence of coordinates $\{ (x_{a_i}, y_{a_i}) \}$ as sample points of a continuous, vector-valued function $\mathbf{r}(p) = (x(p), y(p))$, where $p \in [0, 2\pi]$ represents path parameter for arc length. 

The summation in Eq (1) evaluates the sequence from $i=2$ to $MN$, the topology of $\mathbf{r}(p)$ is strictly defined as open curve. Under uniform convergence together with mesh refinement, polygonal interpolation provides a recovery sequence whose lengths converge to the length of the limiting rectifiable curve \cite{Deng2026}. By Riemann sum, the limit converges to the total arc length. The discrete objective becomes continuous functional:
\begin{equation}
    J[\mathbf{r}] = \|\mathbf{r}(0)\| + \int_{0}^{2\pi} \|\mathbf{r}'(p)\| \, dp.
\end{equation}

The combinatorial constraint in Eq (1) requires that for any rotation angle $\theta_k$, the sequence must intersect one of the three lines forming the triangle boundary of $\mathcal{T}$. We map this to a continuous projection.

Let $L_1$, $L_2$, and $L_3$ denote the lengths of the left, right, and bottom edges. We define the bottom edge $L_3 = 1$. The other edge lengths are
\begin{equation}
    L_1 = \frac{\sin\beta}{\sin(\alpha+\beta)}, \quad L_2 = \frac{\sin\alpha}{\sin(\alpha+\beta)}.
\end{equation}

In the continuum limit, the discrete rotations $\theta_k$ become a continuous phase $t \in [0, 2\pi)$. The outward normal vectors of $\mathcal{T}$, parameterized by rotation $t$ occur at the projection angles:
\begin{equation}
    \phi_1(t) = t + \pi + \alpha, \quad \phi_2(t) = t + \pi - \beta, \quad \phi_3(t) = t.
\end{equation}
The orthogonal distances from the origin to the boundaries are $p_1 = 0$, $p_3 = 0$, and $p_2 = \sin\beta$.

We define the support function \cite{Ball2019} $h(\phi)$ used before in infinite unit strip in \cite{Deng2026-1} for continuous curve $\mathbf{r}(p)$:
\begin{equation}
    h(\phi) = \max_{p \in [0,2\pi]} \left( \mathbf{r}(p) \cdot \hat{\mathbf{n}}(\phi) \right),
\end{equation}
where $\hat{\mathbf{n}}(\phi) = (\cos\phi, \sin\phi)$. 

With support functions of $\mathbf{r}(p)$, its projection onto any vector $\hat{\mathbf{n}}(\phi)$ yields a continuous scalar mapping. By the Intermediate Value Theorem, it forms a connected, gapless interval corresponding identically to the shadow of the curve's convex hull. Hence, if the transformed triangle is $ Q = \{ \mathbf{q} : \mathbf{u}_j \cdot \mathbf{q} \le d_j, \ j = 1, 2, 3 \}, d_j \ge 0, $ and the path begins at $\mathbf{0} \in Q$, then $ h_\Gamma(\mathbf{u}_j) \ge d_j $ implies that the path reaches a point outside the interior of the $j$-th half-space. By continuity, the path algebraically guarantees to intersect boundary lines $\partial Q$.

Thus, the edge-length weights arise from the equilibrium identity $ L_1\mathbf{n}_1 + L_2\mathbf{n}_2 + L_3\mathbf{n}_3 = \mathbf{0}.$ For every starting point $\mathbf{s} \in \mathcal{T}$, the side slacks satisfy
\[L_1d_1(\mathbf s)+L_2d_2(\mathbf s)+L_3d_3(\mathbf s) = \frac{\sin\alpha\sin\beta}{\sin(\alpha+\beta)} = 2\,\operatorname{Area}(\mathcal T). \]

Thus, the support functions must satisfy the linear combination of edge lengths and domain offsets
\begin{equation}
    L_1 h(\phi_1(t)) + L_2 h(\phi_2(t)) + L_3 h(\phi_3(t)) \ge L_1 p_1 + L_2 p_2 + L_3 p_3.
\end{equation}
Substituting lengths and distances, we obtain the constraint
\begin{equation}
    \sin\beta \, h(t+\pi+\alpha) + \sin\alpha \, h(t+\pi-\beta) + \sin(\alpha+\beta) \, h(t) \ge \sin\alpha \sin\beta.
\end{equation}

Synthesizing the arc length functional Eq (2) with the support constraints Eq (7), we completely eliminate the combinatorial complexity constraints in Eq (1). The globally generalized continuous optimization formulation becomes
\begin{equation}
\begin{aligned}
    \text{minimize}_{\mathbf{r}(p)} \quad & \|\mathbf{r}(0)\| + \int_{0}^{2\pi} \|\mathbf{r}'(p)\| \, dp \\
    \text{s.t.} \quad & \sin\beta \, h(t + \pi + \alpha) + \sin\alpha \, h(t + \pi - \beta) + \sin(\alpha + \beta) \, h(t) \ge \sin\alpha \sin\beta, \\
    & h(t) \ge 0, \forall t \in [0, 2\pi)
\end{aligned}
\end{equation}
where $\mathbf{r}(p)$ is an open curve. 

Eq (1) is the general discrete optimization formulation. As we used support function, Eq (8) can be used for every continuous rectifiable path beginning at the origin, including nonconvex curves and nonconvex polygonal chains. This variational architecture allows arbitrary topological scaling, as this formulation removes the explicit permutation and position variables in Eq (1). But the support constraints remain nonsmooth with sums of maxima of functions. Numerical global optimality requires a certified global solver, or matching lower and upper bounds. 

With finite difference method for $[0, 2\pi)$ with $M$ parts in Eq (8), a piecewise-linear path discretization together with angular collocation gives the following finite-dimensional relaxation. Enforcing the constraint at $ t_j = \frac{2\pi j}{M} $, the formula is

\begin{equation}
\begin{aligned}
\text{minimize}_{\{x_i,y_i\}_{i=0}^{N}} 
& \sqrt{x_0^2+y_0^2} +\sum_{i=0}^{N-1}
\sqrt{(x_{i+1}-x_i)^2+(y_{i+1}-y_i)^2}\\
\mathrm{s.t.}& \sin\beta\, \max_{0\le i\le N} \left[
x_i\cos(t_j+\pi+\alpha) +y_i\sin(t_j+\pi+\alpha) \right] \\
&+\sin\alpha\, \max_{0\le i\le N} \left[ x_i\cos(t_j+\pi-\beta)
+y_i\sin(t_j+\pi-\beta) \right] \\
&+\sin(\alpha+\beta)\, \max_{0\le i\le N} \left[ x_i\cos t_j+y_i\sin t_j
\right] \ge \sin\alpha\sin\beta\\
& \max_{0\le i\le N} \left[ x_i\cos t_j+y_i\sin t_j
\right]\ge 0, j=0,\ldots,M-1,\\
& t_j=\frac{2\pi j}{M}
\end{aligned}
\end{equation}
Consequently, the collocation problems converge as $M \to \infty$. This finite difference form can be solved highly efficiently via global optimization algorithm.

\begin{corollary}[Isosceles triangle covering curves]
From Eq (8), the continuous optimization for escaping isosceles triangle with only one base angle $\alpha$ is
\begin{equation}
\begin{aligned}
    \text{minimize}_{\mathbf{r}(p)} \quad &  \|\mathbf{r}(0)\| + \int_{0}^{2\pi} \|\mathbf{r}'(p)\| \, dp \\
    \text{s.t.} \quad & h(t + \pi + \alpha) + h(t + \pi - \alpha) + 2\cos\alpha \, h(t) \ge \sin\alpha, \quad \forall t \in [0, 2\pi), \\
    & h(t) \ge 0, \forall t \in [0, 2\pi)
\end{aligned}
\end{equation}

Recently, Temerev's paper \cite{Temerev2026} also obtained the same constraints.
\end{corollary}

\begin{corollary}[Triangle covering closed curve]
If the curve is closed, then based on the formula from our previous paper \cite{Deng2026}, we simply add a term to the objective function representing the distance between the curve's end and origin, or add a constraint requiring the end and start points to be identical. The formula for closed curve similar to Eq (8) is
\begin{equation}
\begin{aligned}
    \text{minimize}_{\mathbf{r}(p)} \quad &  \|\mathbf{r}(0)\| + \int_{0}^{2\pi} \|\mathbf{r}'(p)\| \, dp+\|\mathbf{r}(2\pi)\| \\
    \text{s.t.} \quad & \sin\beta \, h(t + \pi + \alpha) + \sin\alpha \, h(t + \pi - \beta) + \sin(\alpha + \beta) \, h(t) \ge \sin\alpha \sin\beta, \\
    & h(t) \ge 0,  \forall t \in [0, 2\pi),
\end{aligned}
\end{equation}

\end{corollary}

\subsection{Derivation of formula for covering polygonal chain}

The previous Eq (8) yields results for any curve. Many previous studies found the optimal path with polylines, particularly non-convex zigzag polylines in many cases like equilateral triangle \cite{Besicovitch1965} and some isosceles triangles with large base angle \cite{Gibbs2016}. Especially, the cases with polygonal chains have been previously studied by John Wetzel and others \cite{Sroysang2008} \cite{Füredi2011} \cite{Panraksa2021}. The purpose of this subsection is therefore not to derive different coverage condition, but to restrict the continuous variational problem to the finite-dimensional class of polygonal chain paths. Thus, we derive formula for polygonal path curves. 

First, to clarify the correspondence, we list the forest and worm problems corresponding to different path curves (worms) in Table 1.

\begin{table}[h!]
\centering
\begin{tabular}{|p{7cm}|p{7cm}|}
\hline
\textbf{Description of forest problem} & \textbf{Description of corresponding worm problem} \\ \hline
Lost in a forest of any shape with shortest escape path & Covering worm with any shape \\ \hline
Lost in arbitrary triangle forest with shortest escape path & Covering worm with arbitrary triangle \\ \hline
Lost in arbitrary triangle forest with shortest polygonal chain escape path & Covering polygonal chain with arbitrary \\ \hline
Lost in triangle with shortest closed polygonal chain escape path & Covering a closed polygonal chain with triangle \\ \hline
\end{tabular}
\caption{Description of forest problem and corresponding worm problem}
\end{table}

Then for finite polygonal chain, let $\mathbf{r}_i=(x_i,y_i)\in\mathbb{R}^2,
i=1,\ldots,K$ denote the vertices of an open polygonal curve. Including the initial connection from the origin to the first vertex, the polygonal path is
$
\mathbf{0}\rightarrow \mathbf{r}_1
\rightarrow \mathbf{r}_2
\rightarrow \cdots
\rightarrow \mathbf{r}_K
$.

For a prescribed angular discretization parameter \(M\), we define
\[
t_j=\frac{2\pi j}{M},
j=0,1,\ldots,M-1
\]
The discrete support function generated by the polygonal vertices is
\[
h_K(\theta)
=
\max_{1\leq i\leq K}
\left\{
x_i\cos\theta+y_i\sin\theta
\right\}
=
\max_{1\leq i\leq K}
\left\{
\mathbf{r}_i\cdot
\begin{pmatrix}
\cos\theta\\
\sin\theta
\end{pmatrix}
\right\}.
\]

The finite-dimensional optimization of polygonal chain is
\begin{equation}
\begin{aligned}
\text{minimize}_{\{x_i,y_i\}_{i=1}^{K}}
&\sqrt{x_1^2+y_1^2}+\sum_{i=2}^{K}\sqrt{(x_i-x_{i-1})^2+(y_i-y_{i-1})^2}\\
\mathrm{s.t.} &\sin\beta\,h_K(t_j+\pi+\alpha)+
\sin\alpha\,h_K(t_j+\pi-\beta)+\sin(\alpha+\beta)\,h_K(t_j)\\
&\geq\sin\alpha\,\sin\beta,\forall j=0,\ldots,M-1,\\
& h_K(t_j)\geq 0, \forall j=0,\ldots,M-1.
\end{aligned}
\end{equation}

As \(K\to\infty\) with more segments, polygonal interpolation is dense in the class of rectifiable continuous paths under uniform convergence. A small dilation restores exact support feasibility when interpolation produces an $o(1)$ constraint deficit. Hence the optimal polygonal values converge to the optimal continuous value!

\begin{corollary}[Covering closed polygonal chain]
If the polygonal curve is closed, then we simply add a term to the objective function representing the distance between the curve's end and origin. The curve becomes a closed polygonal chain based at the origin. The formula is
\begin{equation}
\begin{aligned}
\text{minimize}_{\{x_i,y_i\}_{i=1}^{K}}
&
\sqrt{x_1^2+y_1^2}+\sum_{i=2}^{K}\sqrt{(x_i-x_{i-1})^2+(y_i-y_{i-1})^2}
+\sqrt{x_K^2+y_K^2}\\
\mathrm{s.t.} &\sin\beta\,h_K(t_j+\pi+\alpha)+
\sin\alpha\,h_K(t_j+\pi-\beta)+\sin(\alpha+\beta)\,h_K(t_j)\\
&\geq\sin\alpha\,\sin\beta,\forall j=0,\ldots,M-1,\\
& h_K(t_j)\geq 0, \forall j=0,\ldots,M-1.
\end{aligned}
\end{equation}
\end{corollary}

\subsection{Proof of equivalence, existence, and convergence}
In our previous paper \cite{Deng2026} on general solution of Bellman's lost-in-a-forest problem, we have already proved the equivalence, certificate, and convergence of the discretized general TSPN formulation for any shape. Here in this subsection, we prove for the optimal formulations in last two subsections for triangle forest.

For the support function, it must be taken over the complete path. Thus, define
\[
\widetilde{\Gamma}_{\mathbf r} = \{\lambda\mathbf r(0):0\leq\lambda\leq1\} \cup \{\mathbf r(p):0\leq p\leq2\pi\},
\]
\[
h_{\mathbf r}(\phi) = \max_{\mathbf q\in\widetilde{\Gamma}_{\mathbf r}} \mathbf q\cdot(\cos\phi,\sin\phi).
\]
Then $h_{\mathbf r}(\phi)\geq0$ automatically, and the objective function in Eq (2) is precisely the length of the complete path. Set 
\[
\sin\alpha\sin\beta=c
\]
\[
F_{\boldsymbol{\gamma}}(t) = \sin\beta \, h_{\boldsymbol{\gamma}}(t+\pi+\alpha) + \sin\alpha \, h_{\boldsymbol{\gamma}}(t+\pi-\beta) + \sin(\alpha+\beta) \, h_{\boldsymbol{\gamma}}(t).
\]

\begin{theorem}[Exact support function constraint characterization]
A continuous path starting at the origin escapes the triangle for every
unknown starting position and every orientation if and only if
\begin{equation}
F_{\boldsymbol{\gamma}}(t)\geq c, \forall t\in[0,2\pi).
\end{equation}
\end{theorem}

\begin{proof}
Fix $t$. For a starting point $\mathbf s$ in the triangle, let
$(d_1(\mathbf s),d_2(\mathbf s),d_3(\mathbf s))$ be the nonnegative
orthogonal distances from the origin to the three supporting lines of the
translated and rotated triangle. The triangle geometry gives
\begin{equation}
\sin\beta \, d_1(\mathbf s) + \sin\alpha \, d_2(\mathbf s) + \sin(\alpha+\beta) \, d_3(\mathbf s) = c.
\end{equation}
Moreover, as $\mathbf s$ ranges over the triangle, the distance vector
ranges over 
\[
\mathcal D = \left\{ (d_1,d_2,d_3)\in\mathbb R_+^3: \sin\beta \, d_1 + \sin\alpha \, d_2 + \sin(\alpha+\beta) \, d_3 = c \right\}.
\]

Let 
$
H_1=h_{\boldsymbol{\gamma}}(t+\pi+\alpha),H_2=h_{\boldsymbol{\gamma}}(t+\pi-\beta),H_3=h_{\boldsymbol{\gamma}}(t).
$
Since the path starts at the origin, it escapes the transformed triangle with
distance vector $(d_1,d_2,d_3)$ exactly when
\[
\max_{j=1,2,3}(H_j-d_j)\geq0.
\]

Suppose first that
\[
\sin\beta \, H_1 + \sin\alpha \, H_2 + \sin(\alpha+\beta) \, H_3 \geq c.
\]
It is impossible to have $H_j<d_j$ for every $j$, because multiplying
these inequalities by the positive weights and summing would contradict
Eq (15). Hence the path escapes every
translated triangle.

Conversely, suppose that the weighted support inequality fails. Define
\[
\delta = \frac{c - \sin\beta \, H_1 - \sin\alpha \, H_2 - \sin(\alpha+\beta) \, H_3}{\sin\beta+\sin\alpha+\sin(\alpha+\beta)} > 0
\]
and let $d_j=H_j+\delta, j=1,2,3.$ Then $(d_1,d_2,d_3)\in\mathcal D$, but $H_j<d_j$ for every $j$.
The corresponding starting position therefore produces a transformed
triangle that contains the entire path in its interior. Thus escape
fails.

The argument uses only continuity of the path and the attainment of each
support maximum. 
\end{proof}

\begin{theorem}[Existence of optimal escape path for triangle]
The continuous optimization problem, escape triangle with shortest escape path admits a minimizer.
\end{theorem}

\begin{proof}
Since triangle is a closed shape, and Finch's paper \cite{Finch2004} proved that the diameter of a closed shape is necessarily an escape path. Thus, an optimal continuous escape path for the triangle must exist. Therefore, the continuous optimization problem
\[
v = \inf \left\{ \ell(\boldsymbol{\gamma}): \boldsymbol{\gamma}(0)=\mathbf0, \; F_{\boldsymbol{\gamma}}(t)\geq c \ \forall t \right\}
\]
admits a minimizer.
\end{proof}

For the polygonal formulation, set $\mathbf q_0=\mathbf0$ and use
\[
h_K(\phi) = \max_{0\leq i\leq K} \mathbf q_i\cdot(\cos\phi,\sin\phi).
\]
This is the support function of the entire polygonal path, as a linear
functional attains its maximum on a line segment at one of its endpoints.
Let $v_K$ denote the minimum of the exact polygonal problem in which the
support constraint is imposed for every $t\in[0,2\pi)$.

\begin{theorem}[Convergence of polygonal paths]
The exact polygonal optimal values satisfy
\begin{equation}
v_K\rightarrow v \text{as }K\to\infty.
\end{equation}
Moreover, every bounded sequence of asymptotically optimal polygonal paths
a uniformly convergent subsequence whose limit is optimal continuous path.
\end{theorem}

\begin{proof}
Since polygonal chain paths are admissible continuous paths, 
\[
v\leq v_K.
\]

Let $\boldsymbol{\gamma}^*$ be an optimal continuous path and let $P_K$
be its polygonal interpolation on a partition whose mesh tends to zero. Then
\[
\|P_K-\boldsymbol{\gamma}^*\|_\infty = \varepsilon_K\to0, \qquad \ell(P_K)\leq\ell(\boldsymbol{\gamma}^*)=v.
\]
Define $W = \sin\beta+\sin\alpha+\sin(\alpha+\beta).$ Uniform support convergence gives
\[
F_{P_K}(t)\geq c-W\varepsilon_K.
\]
For sufficiently large $K$, set $\rho_K = \frac{c}{c-W\varepsilon_K}. $ Then $\rho_K\to1$, and positive homogeneity of the support function yields
\[
F_{\rho_KP_K}(t) = \rho_KF_{P_K}(t) \geq c.
\]
Thus $\rho_KP_K$ is an exactly feasible polygonal recovery sequence and
\[
v_K \leq \rho_K\ell(P_K) \leq \rho_Kv.
\]
Therefore
\[
\limsup_{K\to\infty}v_K\leq v,
\]
which, together with $v\leq v_K$, proves Eq (16).

Compactness and convergence of asymptotically optimal polygonal paths follow
from the same Arzel\`a--Ascoli and lower-semicontinuity argument used in
Theorem 5. Similar approach was used in proof of Theorem 2 and 7 in \cite{Deng2026}.
\end{proof}

Let $v_{K,M}$ be the polygonal problem in which the support constraint is
enforced only at $t_j=\frac{2\pi j}{M}, j=0,\ldots,M-1.$

\begin{lemma}[Angular collocation error]
If escape path starts at the origin and has length at most $B$, then
\begin{equation}
|F_{\boldsymbol{\gamma}}(t)-F_{\boldsymbol{\gamma}}(s)| \leq WB \, \operatorname{dist}_{\mathbb S^1}(t,s).
\end{equation}
Consequently, satisfaction of the constraints at all collocation points
implies
\begin{equation}
F_{\boldsymbol{\gamma}}(t) \geq c-\frac{\pi WB}{M}, \qquad \forall t\in[0,2\pi).
\end{equation}
\end{lemma}

\begin{proof}
Since the path starts at the origin, every point on it has norm at most
$B$. Therefore
\[
|h_{\boldsymbol{\gamma}}(\phi) -h_{\boldsymbol{\gamma}}(\psi)| \leq B \, \operatorname{dist}_{\mathbb S^1}(\phi,\psi).
\]
Applying this estimate to the three terms defining
$F_{\boldsymbol{\gamma}}$ proves
Eq (17). Every angle lies within
$\pi/M$ of a collocation point, which gives Eq (18).
\end{proof}

\begin{theorem}[Convergence of the fully discrete formulation]
For each fixed $K$, $v_{K,M}\rightarrow v_K \text{as }M\to\infty.$ Furthermore, $K_n\to\infty$ and $M_n\to\infty$, then
\begin{equation}
v_{K_n,M_n}\rightarrow v.
\end{equation}
\end{theorem}

\begin{proof}
Exact angular feasibility implies collocation feasibility, so
\[
v_{K,M}\leq v_K.
\]
Conversely, choose nearly optimal collocation solutions with uniformly
bounded objective values. Their vertices lie in a common bounded set.
After passing to a subsequence, the vertices and the associated support
functions converge uniformly. By Eq (18), their maximum constraint
violation tends to zero as $M\to\infty$; hence the limiting polygonal path
is feasible for every angle. Lower semicontinuity of the polygonal length
then gives
\[
v_K \leq \liminf_{M\to\infty}v_{K,M}.
\]
Thus $v_{K,M}\to v_K$.

For $K_n,M_n\to\infty$, compactness, the preceding collocation estimate,
and lower semicontinuity give
\[
v \leq \liminf_{n\to\infty}v_{K_n,M_n}.
\]
The polygonal recovery sequence constructed in Theorem 6 is feasible for all angles and hence for every collocation grid, yielding
\[
\limsup_{n\to\infty}v_{K_n,M_n}\leq v.
\]
Therefore Eq(19) follows.
\end{proof}

The proofs provide the compactness and liminf inequality, together
with an explicit polygonal recovery sequence. Hence the polygonal and
discretized formulations converge to the continuous
optimization problem.

\subsection{Extend to universal polygon covering curve and polygonal chain}
A direct corollary and generalization of Eq(8) is that, by applying a similar process to multiple lines (boundary of polygons), such as in the "escape from polygon" or "polygon covering curves" rather than triangles, we can derive optimization formulas.

For a polygon defined by $m$ lines,
\[
    P = \{q : u_j \cdot q \le b_j, \ j = 1, \dots, m\},
\]
the general continuous optimization formula of covering convex escape curve is
\begin{equation}
\begin{aligned}
\text{minimize}_{r \in AC([0,2\pi];\mathbb{R}^2)} \quad & \|r(0)\| + \int_0^{2\pi} \|r'(p)\| \, dp, \\
\text{s.t.} \quad & \max_{\substack{\lambda_j \ge 0 \\ \sum_{j=1}^m \lambda_j = 1 \\ \sum_{j=1}^m \lambda_j u_j = 0}} \sum_{j=1}^m \lambda_j [h_r(t + \phi_j) - b_j] \ge 0, \\
& \forall t \in [0, 2\pi), \\
& h_r(t) = \max \left\{ 0, \max_{p \in [0,1]} r(p) \cdot (\cos t, \sin t) \right\}.
\end{aligned}
\end{equation}

We can also derive optimization formulas for polygon covering polygonal chain curves from Eq (12) and Eq (13) as
\begin{equation}
\begin{aligned}
\text{minimize}_{\{x_i,y_i\}_{i=1}^{K}}
&\sqrt{x_1^2+y_1^2}+\sum_{i=2}^{K}\sqrt{(x_i-x_{i-1})^2+(y_i-y_{i-1})^2}\\
\mathrm{s.t.} &\max_{\substack{\lambda_j\ge 0,\; j=1,\ldots,m\\
\sum_{j=1}^{m}\lambda_j=1\\
\sum_{j=1}^{m}\lambda_j u_j=0}}
\sum_{j=1}^{m}\lambda_j\left[h_K(t+\phi_j)-b_j\right]
\ge 0,\forall t\in[0,2\pi),\\
&
h_K(\theta)
=
\max\left\{
0,\,
\max_{1\le i\le K}
\left[
x_i\cos\theta+y_i\sin\theta
\right]
\right\}.
\end{aligned}
\end{equation}

The proof of above formulas is very similar to that of the previous subsection and will not be repeated again.

\section{Results for arbitrary triangle cover}

\subsection{Triangle covering curves}
Figure 2 shows the numerical results of triangle covering curve. We solved the optimization in Eq (9) by discretizing interval $[0, 2\pi]$ into 1440 parts. The optimization was terminated at iteration \(k\) when the relative change in the decision vector
\(\mathbf{z}^{(k)}=(x_0^{(k)},y_0^{(k)},\ldots,x_N^{(k)},y_N^{(k)})\) satisfied
\[
\left\|\mathbf{z}^{(k)}-\mathbf{z}^{(k-1)}\right\|_2
\leq
\max\left\{10^{-9},\,10^{-8}\left\|\mathbf{z}^{(k)}\right\|_2\right\},
\]
and both the maximum constraint violation and the generalized KKT stationarity residual were less than \(10^{-9}\). The generalized KKT residual was employed because the Euclidean-norm objective function and the maximum constraints are nonsmooth.

The figure shows the results when base angles are multiples of 5 degrees, and $0<\alpha \le \beta \le \pi - \alpha - \beta$, without duplications. For polygonal chain, we calculated no more than six segments.  We selected the shortest result of each triangle for curve and polygonal chain to show in the figure.

For an isosceles triangle, we can obtain the same results of various escape path shape as in Gibbs's paper \cite{Gibbs2016}. For shape of arbitrary triangles especially non-isosceles ones, we obtain all the results for the first time.

\begin{figure}[p]
\centering
\begin{subfigure}[b]{0.326\linewidth}
\centering
\includegraphics[width=1\linewidth]{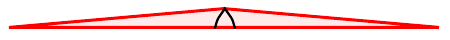}
\caption*{$5^{\circ}-5^{\circ}$}
\end{subfigure}
\begin{subfigure}[b]{0.326\linewidth}
\centering
\includegraphics[width=1\linewidth]{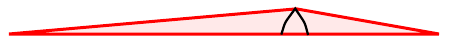}
\caption*{$5^{\circ}-10^{\circ}$}
\end{subfigure}
\begin{subfigure}[b]{0.326\linewidth}
\centering
\includegraphics[width=1\linewidth]{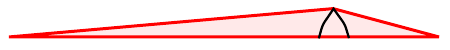}
\caption*{$5^{\circ}-15^{\circ}$}
\end{subfigure}
\begin{subfigure}[b]{0.326\linewidth}
\centering
\includegraphics[width=1\linewidth]{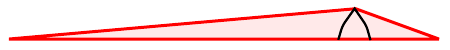}
\caption*{$5^{\circ}-20^{\circ}$}
\end{subfigure}
\begin{subfigure}[b]{0.326\linewidth}
\centering
\includegraphics[width=1\linewidth]{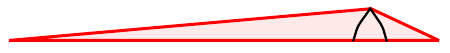}
\caption*{$5^{\circ}-25^{\circ}$}
\end{subfigure}
\begin{subfigure}[b]{0.326\linewidth}
\centering
\includegraphics[width=1\linewidth]{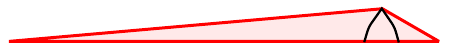}
\caption*{$5^{\circ}-30^{\circ}$}
\end{subfigure}
\begin{subfigure}[b]{0.326\linewidth}
\centering
\includegraphics[width=1\linewidth]{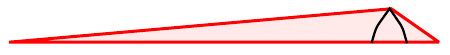}
\caption*{$5^{\circ}-35^{\circ}$}
\end{subfigure}
\begin{subfigure}[b]{0.326\linewidth}
\centering
\includegraphics[width=1\linewidth]{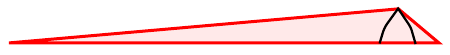}
\caption*{$5^{\circ}-40^{\circ}$}
\end{subfigure}
\begin{subfigure}[b]{0.326\linewidth}
\centering
\includegraphics[width=1\linewidth]{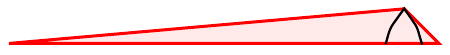}
\caption*{$5^{\circ}-45^{\circ}$}
\end{subfigure}
\begin{subfigure}[b]{0.326\linewidth}
\centering
\includegraphics[width=1\linewidth]{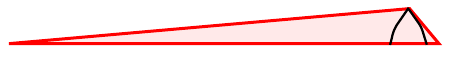}
\caption*{$5^{\circ}-50^{\circ}$}
\end{subfigure}
\begin{subfigure}[b]{0.326\linewidth}
\centering
\includegraphics[width=1\linewidth]{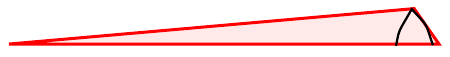}
\caption*{$5^{\circ}-55^{\circ}$}
\end{subfigure}
\begin{subfigure}[b]{0.326\linewidth}
\centering
\includegraphics[width=1\linewidth]{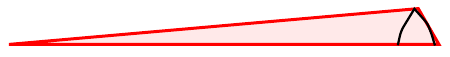}
\caption*{$5^{\circ}-60^{\circ}$}
\end{subfigure}
\begin{subfigure}[b]{0.326\linewidth}
\centering
\includegraphics[width=1\linewidth]{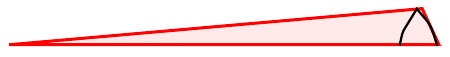}
\caption*{$5^{\circ}-65^{\circ}$}
\end{subfigure}
\begin{subfigure}[b]{0.326\linewidth}
\centering
\includegraphics[width=1\linewidth]{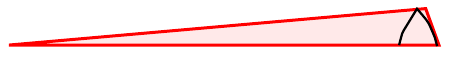}
\caption*{$5^{\circ}-70^{\circ}$}
\end{subfigure}
\begin{subfigure}[b]{0.326\linewidth}
\centering
\includegraphics[width=1\linewidth]{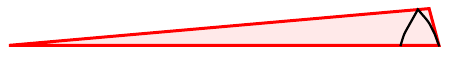}
\caption*{$5^{\circ}-75^{\circ}$}
\end{subfigure}
\begin{subfigure}[b]{0.326\linewidth}
\centering
\includegraphics[width=1\linewidth]{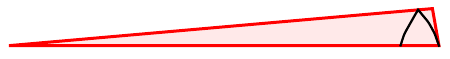}
\caption*{$5^{\circ}-80^{\circ}$}
\end{subfigure}
\begin{subfigure}[b]{0.326\linewidth}
\centering
\includegraphics[width=1\linewidth]{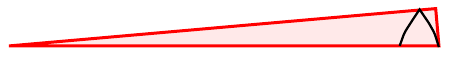}
\caption*{$5^{\circ}-85^{\circ}$}
\end{subfigure}
\begin{subfigure}[b]{0.326\linewidth}
\centering
\includegraphics[width=1\linewidth]{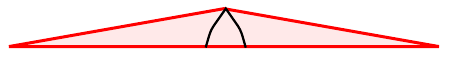}
\caption*{$10^{\circ}-10^{\circ}$}
\end{subfigure}
\begin{subfigure}[b]{0.326\linewidth}
\centering
\includegraphics[width=1\linewidth]{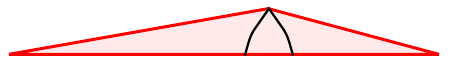}
\caption*{$10^{\circ}-15^{\circ}$}
\end{subfigure}
\begin{subfigure}[b]{0.326\linewidth}
\centering
\includegraphics[width=1\linewidth]{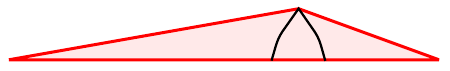}
\caption*{$10^{\circ}-20^{\circ}$}
\end{subfigure}
\begin{subfigure}[b]{0.326\linewidth}
\centering
\includegraphics[width=1\linewidth]{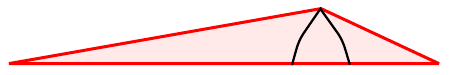}
\caption*{$10^{\circ}-25^{\circ}$}
\end{subfigure}
\begin{subfigure}[b]{0.326\linewidth}
\centering
\includegraphics[width=1\linewidth]{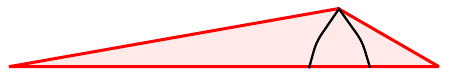}
\caption*{$10^{\circ}-30^{\circ}$}
\end{subfigure}
\begin{subfigure}[b]{0.326\linewidth}
\centering
\includegraphics[width=1\linewidth]{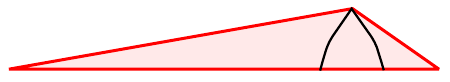}
\caption*{$10^{\circ}-35^{\circ}$}
\end{subfigure}
\begin{subfigure}[b]{0.326\linewidth}
\centering
\includegraphics[width=1\linewidth]{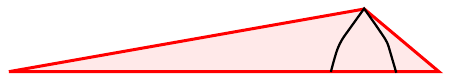}
\caption*{$10^{\circ}-40^{\circ}$}
\end{subfigure}
\begin{subfigure}[b]{0.326\linewidth}
\centering
\includegraphics[width=1\linewidth]{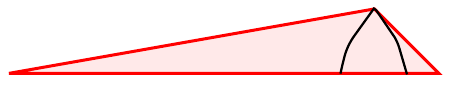}
\caption*{$10^{\circ}-45^{\circ}$}
\end{subfigure}
\begin{subfigure}[b]{0.326\linewidth}
\centering
\includegraphics[width=1\linewidth]{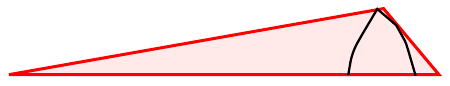}
\caption*{$10^{\circ}-50^{\circ}$}
\end{subfigure}
\begin{subfigure}[b]{0.326\linewidth}
\centering
\includegraphics[width=1\linewidth]{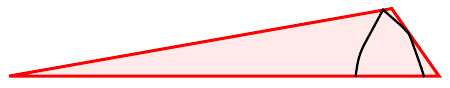}
\caption*{$10^{\circ}-55^{\circ}$}
\end{subfigure}
\begin{subfigure}[b]{0.326\linewidth}
\centering
\includegraphics[width=1\linewidth]{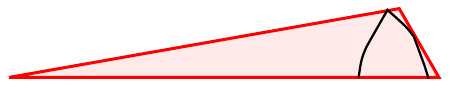}
\caption*{$10^{\circ}-60^{\circ}$}
\end{subfigure}
\begin{subfigure}[b]{0.326\linewidth}
\centering
\includegraphics[width=1\linewidth]{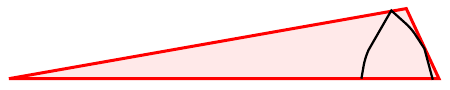}
\caption*{$10^{\circ}-65^{\circ}$}
\end{subfigure}
\begin{subfigure}[b]{0.326\linewidth}
\centering
\includegraphics[width=1\linewidth]{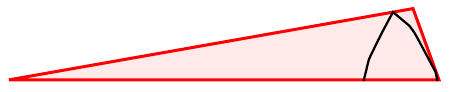}
\caption*{$10^{\circ}-70^{\circ}$}
\end{subfigure}
\begin{subfigure}[b]{0.326\linewidth}
\centering
\includegraphics[width=1\linewidth]{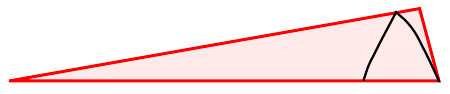}
\caption*{$10^{\circ}-75^{\circ}$}
\end{subfigure}
\begin{subfigure}[b]{0.326\linewidth}
\centering
\includegraphics[width=1\linewidth]{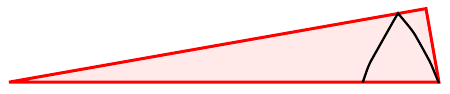}
\caption*{$10^{\circ}-80^{\circ}$}
\end{subfigure}
\begin{subfigure}[b]{0.326\linewidth}
\centering
\includegraphics[width=1\linewidth]{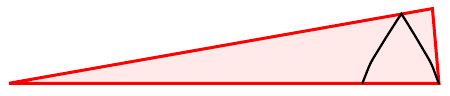}
\caption*{$10^{\circ}-85^{\circ}$}
\end{subfigure}
\begin{subfigure}[b]{0.326\linewidth}
\centering
\includegraphics[width=1\linewidth]{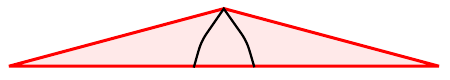}
\caption*{$15^{\circ}-15^{\circ}$}
\end{subfigure}
\begin{subfigure}[b]{0.326\linewidth}
\centering
\includegraphics[width=1\linewidth]{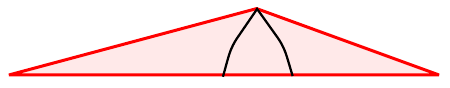}
\caption*{$15^{\circ}-20^{\circ}$}
\end{subfigure}
\begin{subfigure}[b]{0.326\linewidth}
\centering
\includegraphics[width=1\linewidth]{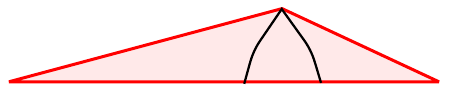}
\caption*{$15^{\circ}-25^{\circ}$}
\end{subfigure}
\caption{Results of triangle covering curve (Black curve is escape path, red curve is forest boundary, and caption is the base angle)}
\end{figure}

\begin{figure}[p]
\ContinuedFloat
\centering
\begin{subfigure}[b]{0.326\linewidth}
\centering
\includegraphics[width=1\linewidth]{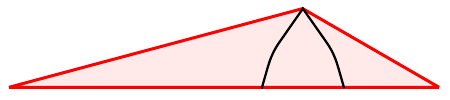}
\caption*{$15^{\circ}-30^{\circ}$}
\end{subfigure}
\begin{subfigure}[b]{0.326\linewidth}
\centering
\includegraphics[width=1\linewidth]{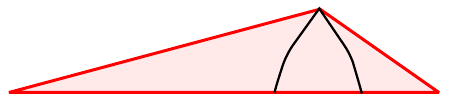}
\caption*{$15^{\circ}-35^{\circ}$}
\end{subfigure}
\begin{subfigure}[b]{0.326\linewidth}
\centering
\includegraphics[width=1\linewidth]{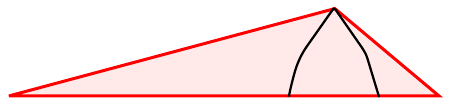}
\caption*{$15^{\circ}-40^{\circ}$}
\end{subfigure}
\begin{subfigure}[b]{0.326\linewidth}
\centering
\includegraphics[width=1\linewidth]{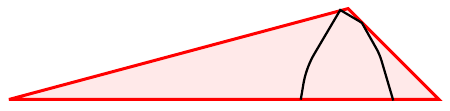}
\caption*{$15^{\circ}-45^{\circ}$}
\end{subfigure}
\begin{subfigure}[b]{0.326\linewidth}
\centering
\includegraphics[width=1\linewidth]{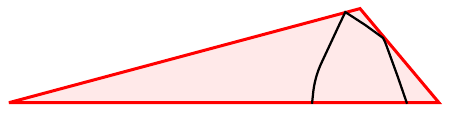}
\caption*{$15^{\circ}-50^{\circ}$}
\end{subfigure}
\begin{subfigure}[b]{0.326\linewidth}
\centering
\includegraphics[width=1\linewidth]{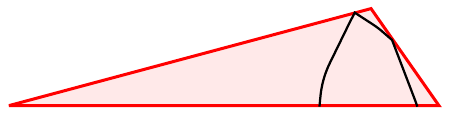}
\caption*{$15^{\circ}-55^{\circ}$}
\end{subfigure}
\begin{subfigure}[b]{0.326\linewidth}
\centering
\includegraphics[width=1\linewidth]{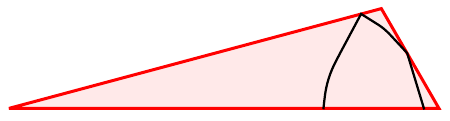}
\caption*{$15^{\circ}-60^{\circ}$}
\end{subfigure}
\begin{subfigure}[b]{0.326\linewidth}
\centering
\includegraphics[width=1\linewidth]{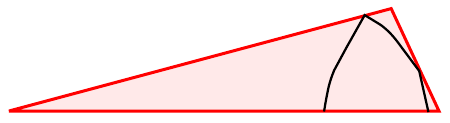}
\caption*{$15^{\circ}-65^{\circ}$}
\end{subfigure}
\begin{subfigure}[b]{0.326\linewidth}
\centering
\includegraphics[width=1\linewidth]{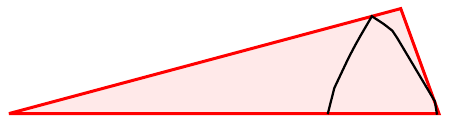}
\caption*{$15^{\circ}-70^{\circ}$}
\end{subfigure}
\begin{subfigure}[b]{0.326\linewidth}
\centering
\includegraphics[width=1\linewidth]{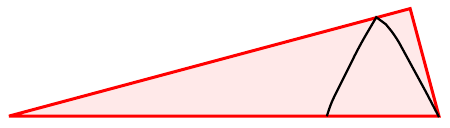}
\caption*{$15^{\circ}-75^{\circ}$}
\end{subfigure}
\begin{subfigure}[b]{0.326\linewidth}
\centering
\includegraphics[width=1\linewidth]{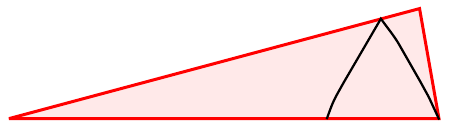}
\caption*{$15^{\circ}-80^{\circ}$}
\end{subfigure}
\begin{subfigure}[b]{0.326\linewidth}
\centering
\includegraphics[width=1\linewidth]{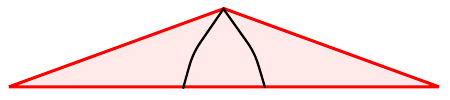}
\caption*{$20^{\circ}-20^{\circ}$}
\end{subfigure}
\begin{subfigure}[b]{0.326\linewidth}
\centering
\includegraphics[width=1\linewidth]{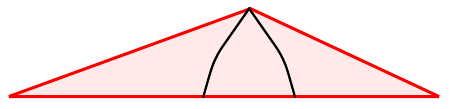}
\caption*{$20^{\circ}-25^{\circ}$}
\end{subfigure}
\begin{subfigure}[b]{0.326\linewidth}
\centering
\includegraphics[width=1\linewidth]{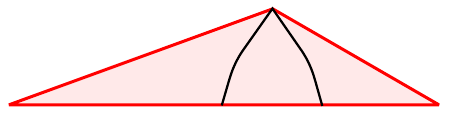}
\caption*{$20^{\circ}-30^{\circ}$}
\end{subfigure}
\begin{subfigure}[b]{0.326\linewidth}
\centering
\includegraphics[width=1\linewidth]{20-25.pdf}
\caption*{$20^{\circ}-25^{\circ}$}
\end{subfigure}
\begin{subfigure}[b]{0.326\linewidth}
\centering
\includegraphics[width=1\linewidth]{20-30.pdf}
\caption*{$20^{\circ}-30^{\circ}$}
\end{subfigure}
\begin{subfigure}[b]{0.326\linewidth}
\centering
\includegraphics[width=1\linewidth]{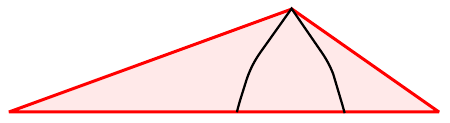}
\caption*{$20^{\circ}-35^{\circ}$}
\end{subfigure}
\begin{subfigure}[b]{0.326\linewidth}
\centering
\includegraphics[width=1\linewidth]{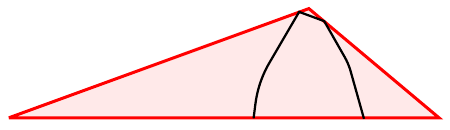}
\caption*{$20^{\circ}-40^{\circ}$}
\end{subfigure}
\begin{subfigure}[b]{0.326\linewidth}
\centering
\includegraphics[width=1\linewidth]{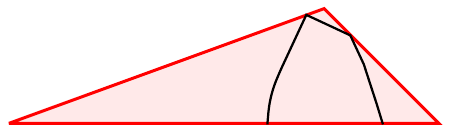}
\caption*{$20^{\circ}-45^{\circ}$}
\end{subfigure}
\begin{subfigure}[b]{0.326\linewidth}
\centering
\includegraphics[width=1\linewidth]{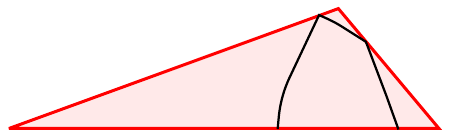}
\caption*{$20^{\circ}-50^{\circ}$}
\end{subfigure}
\begin{subfigure}[b]{0.326\linewidth}
\centering
\includegraphics[width=1\linewidth]{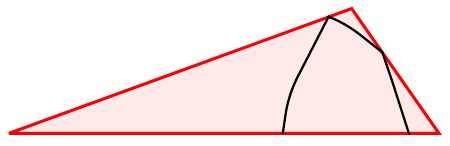}
\caption*{$20^{\circ}-55^{\circ}$}
\end{subfigure}
\begin{subfigure}[b]{0.326\linewidth}
\centering
\includegraphics[width=1\linewidth]{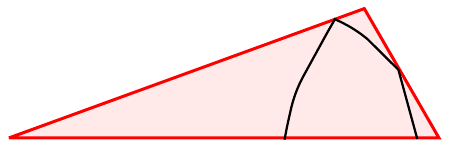}
\caption*{$20^{\circ}-60^{\circ}$}
\end{subfigure}
\begin{subfigure}[b]{0.326\linewidth}
\centering
\includegraphics[width=1\linewidth]{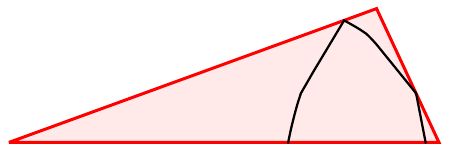}
\caption*{$20^{\circ}-65^{\circ}$}
\end{subfigure}
\begin{subfigure}[b]{0.326\linewidth}
\centering
\includegraphics[width=1\linewidth]{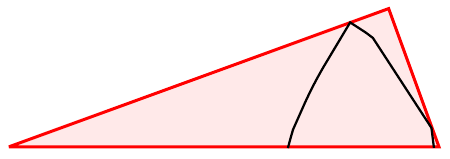}
\caption*{$20^{\circ}-70^{\circ}$}
\end{subfigure}
\begin{subfigure}[b]{0.326\linewidth}
\centering
\includegraphics[width=1\linewidth]{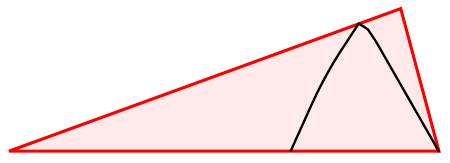}
\caption*{$20^{\circ}-75^{\circ}$}
\end{subfigure}
\begin{subfigure}[b]{0.326\linewidth}
\centering
\includegraphics[width=1\linewidth]{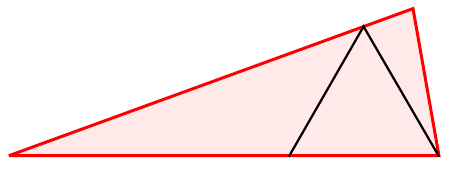}
\caption*{$20^{\circ}-80^{\circ}$}
\end{subfigure}
\begin{subfigure}[b]{0.326\linewidth}
\centering
\includegraphics[width=1\linewidth]{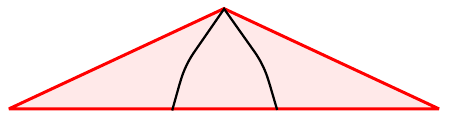}
\caption*{$25^{\circ}-25^{\circ}$}
\end{subfigure}
\caption{Continued results of triangle covering curve (Black curve is escape path, red curve is forest boundary, and caption is the base angle)}
\end{figure}

\begin{figure}[p]
\ContinuedFloat
\centering
\begin{subfigure}[b]{0.326\linewidth}
\centering
\includegraphics[width=1\linewidth]{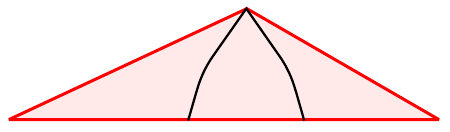}
\caption*{$25^{\circ}-30^{\circ}$}
\end{subfigure}
\begin{subfigure}[b]{0.326\linewidth}
\centering
\includegraphics[width=1\linewidth]{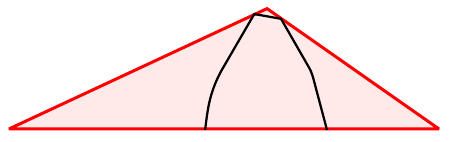}
\caption*{$25^{\circ}-35^{\circ}$}
\end{subfigure}
\begin{subfigure}[b]{0.326\linewidth}
\centering
\includegraphics[width=1\linewidth]{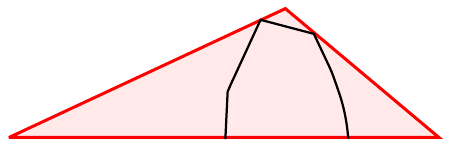}
\caption*{$25^{\circ}-40^{\circ}$}
\end{subfigure}
\begin{subfigure}[b]{0.326\linewidth}
\centering
\includegraphics[width=1\linewidth]{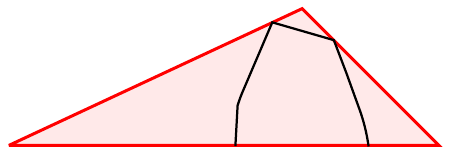}
\caption*{$25^{\circ}-45^{\circ}$}
\end{subfigure}
\begin{subfigure}[b]{0.326\linewidth}
\centering
\includegraphics[width=1\linewidth]{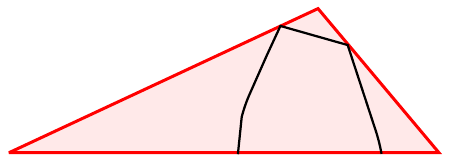}
\caption*{$25^{\circ}-50^{\circ}$}
\end{subfigure}
\begin{subfigure}[b]{0.326\linewidth}
\centering
\includegraphics[width=1\linewidth]{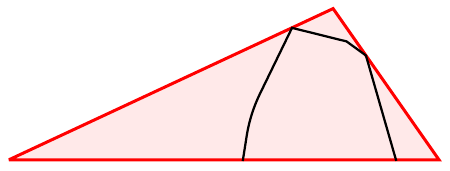}
\caption*{$25^{\circ}-55^{\circ}$}
\end{subfigure}
\begin{subfigure}[b]{0.326\linewidth}
\centering
\includegraphics[width=1\linewidth]{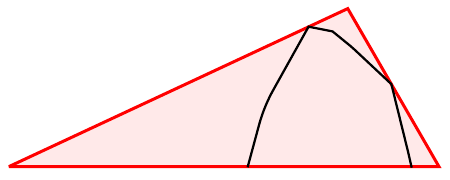}
\caption*{$25^{\circ}-60^{\circ}$}
\end{subfigure}
\begin{subfigure}[b]{0.326\linewidth}
\centering
\includegraphics[width=1\linewidth]{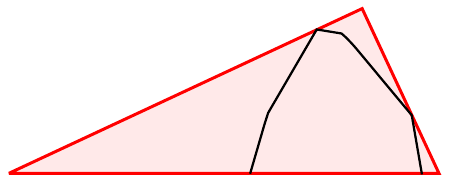}
\caption*{$25^{\circ}-65^{\circ}$}
\end{subfigure}
\begin{subfigure}[b]{0.326\linewidth}
\centering
\includegraphics[width=1\linewidth]{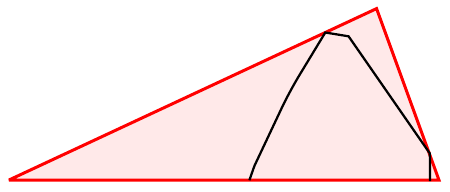}
\caption*{$25^{\circ}-70^{\circ}$}
\end{subfigure}
\begin{subfigure}[b]{0.326\linewidth}
\centering
\includegraphics[width=1\linewidth]{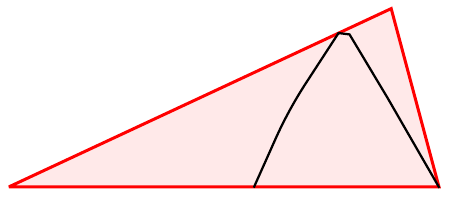}
\caption*{$25^{\circ}-75^{\circ}$}
\end{subfigure}
\begin{subfigure}[b]{0.326\linewidth}
\centering
\includegraphics[width=1\linewidth]{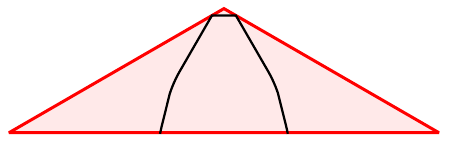}
\caption*{$30^{\circ}-30^{\circ}$}
\end{subfigure}
\begin{subfigure}[b]{0.326\linewidth}
\centering
\includegraphics[width=1\linewidth]{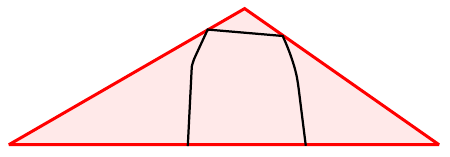}
\caption*{$30^{\circ}-35^{\circ}$}
\end{subfigure}
\begin{subfigure}[b]{0.326\linewidth}
\centering
\includegraphics[width=1\linewidth]{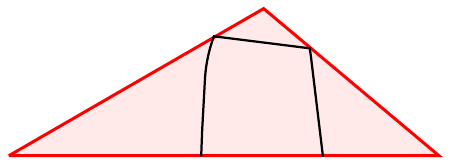}
\caption*{$30^{\circ}-40^{\circ}$}
\end{subfigure}
\begin{subfigure}[b]{0.326\linewidth}
\centering
\includegraphics[width=1\linewidth]{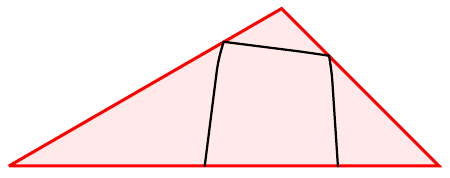}
\caption*{$30^{\circ}-45^{\circ}$}
\end{subfigure}
\begin{subfigure}[b]{0.326\linewidth}
\centering
\includegraphics[width=1\linewidth]{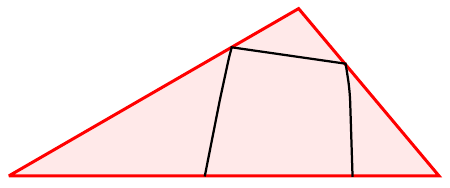}
\caption*{$30^{\circ}-50^{\circ}$}
\end{subfigure}
\begin{subfigure}[b]{0.326\linewidth}
\centering
\includegraphics[width=1\linewidth]{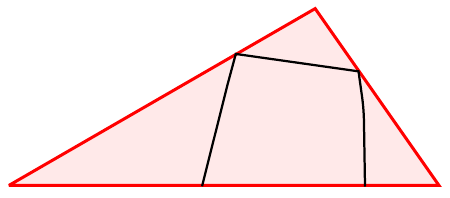}
\caption*{$30^{\circ}-55^{\circ}$}
\end{subfigure}
\begin{subfigure}[b]{0.326\linewidth}
\centering
\includegraphics[width=1\linewidth]{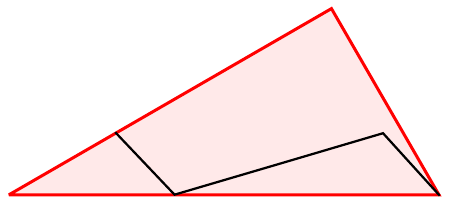}
\caption*{$30^{\circ}-60^{\circ}$}
\end{subfigure}
\begin{subfigure}[b]{0.326\linewidth}
\centering
\includegraphics[width=1\linewidth]{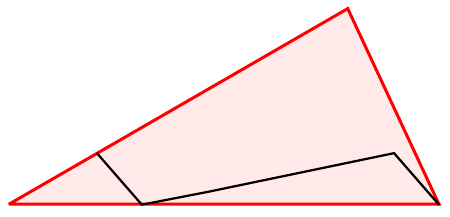}
\caption*{$30^{\circ}-65^{\circ}$}
\end{subfigure}

\caption{Continued results of triangle covering curve (Black curve is escape path, red curve is forest boundary, and caption is the base angle)}
\end{figure}

\begin{figure}[p]
\ContinuedFloat
\centering
\begin{subfigure}[b]{0.326\linewidth}
\centering
\includegraphics[width=1\linewidth]{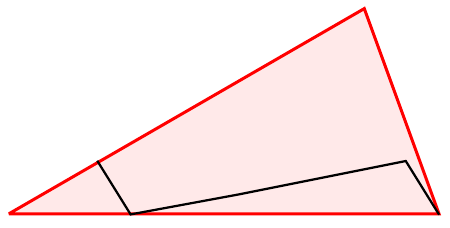}
\caption*{$30^{\circ}-70^{\circ}$}
\end{subfigure}
\begin{subfigure}[b]{0.326\linewidth}
\centering
\includegraphics[width=1\linewidth]{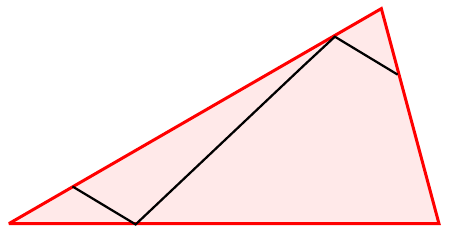}
\caption*{$30^{\circ}-75^{\circ}$}
\end{subfigure}
\begin{subfigure}[b]{0.326\linewidth}
\centering
\includegraphics[width=1\linewidth]{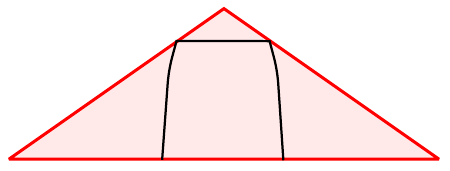}
\caption*{$35^{\circ}-35^{\circ}$}
\end{subfigure}
\begin{subfigure}[b]{0.326\linewidth}
\centering
\includegraphics[width=1\linewidth]{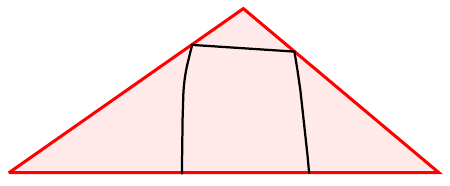}
\caption*{$35^{\circ}-40^{\circ}$}
\end{subfigure}
\begin{subfigure}[b]{0.326\linewidth}
\centering
\includegraphics[width=1\linewidth]{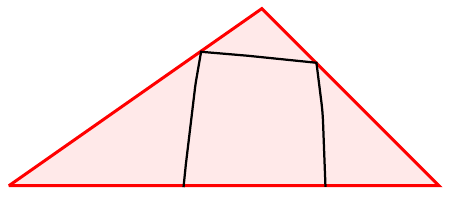}
\caption*{$35^{\circ}-45^{\circ}$}
\end{subfigure}
\begin{subfigure}[b]{0.326\linewidth}
\centering
\includegraphics[width=1\linewidth]{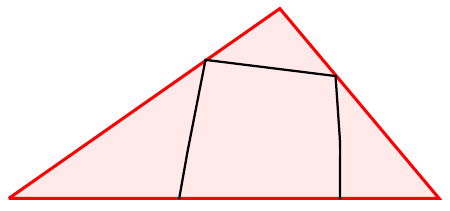}
\caption*{$35^{\circ}-50^{\circ}$}
\end{subfigure}
\begin{subfigure}[b]{0.326\linewidth}
\centering
\includegraphics[width=1\linewidth]{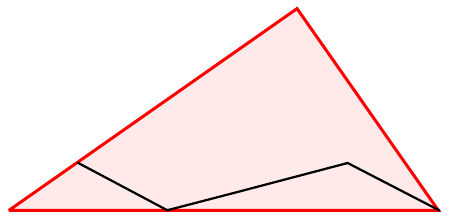}
\caption*{$35^{\circ}-55^{\circ}$}
\end{subfigure}
\begin{subfigure}[b]{0.326\linewidth}
\centering
\includegraphics[width=1\linewidth]{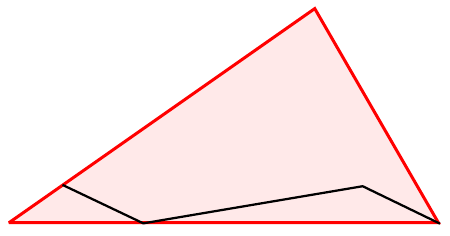}
\caption*{$35^{\circ}-60^{\circ}$}
\end{subfigure}
\begin{subfigure}[b]{0.326\linewidth}
\centering
\includegraphics[width=1\linewidth]{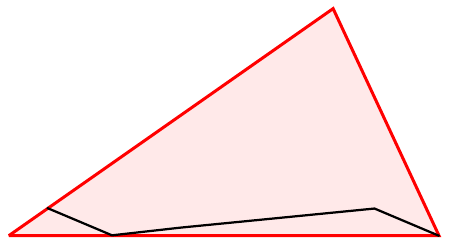}
\caption*{$35^{\circ}-65^{\circ}$}
\end{subfigure}
\begin{subfigure}[b]{0.326\linewidth}
\centering
\includegraphics[width=1\linewidth]{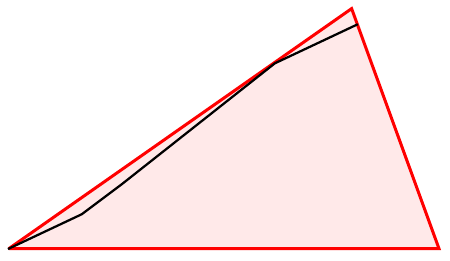}
\caption*{$35^{\circ}-70^{\circ}$}
\end{subfigure}
\begin{subfigure}[b]{0.326\linewidth}
\centering
\includegraphics[width=1\linewidth]{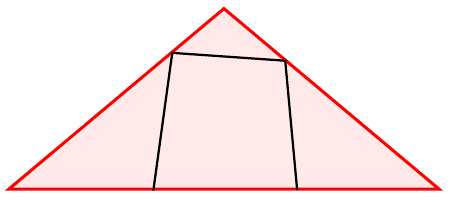}
\caption*{$40^{\circ}-40^{\circ}$}
\end{subfigure}
\begin{subfigure}[b]{0.326\linewidth}
\centering
\includegraphics[width=1\linewidth]{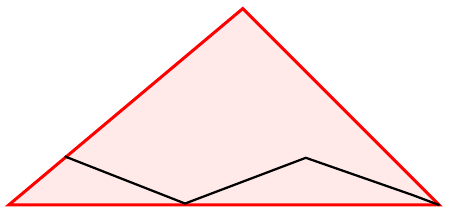}
\caption*{$40^{\circ}-45^{\circ}$}
\end{subfigure}
\begin{subfigure}[b]{0.326\linewidth}
\centering
\includegraphics[width=1\linewidth]{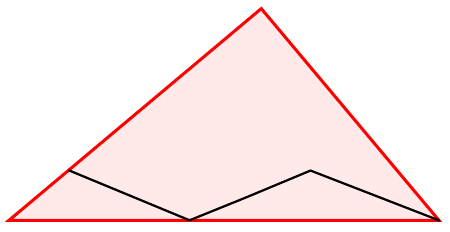}
\caption*{$40^{\circ}-50^{\circ}$}
\end{subfigure}
\begin{subfigure}[b]{0.326\linewidth}
\centering
\includegraphics[width=1\linewidth]{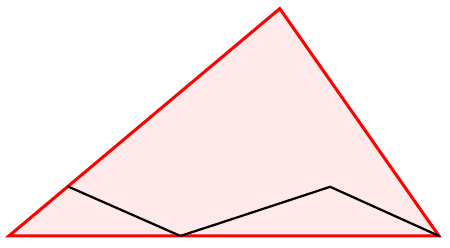}
\caption*{$40^{\circ}-55^{\circ}$}
\end{subfigure}
\begin{subfigure}[b]{0.326\linewidth}
\centering
\includegraphics[width=1\linewidth]{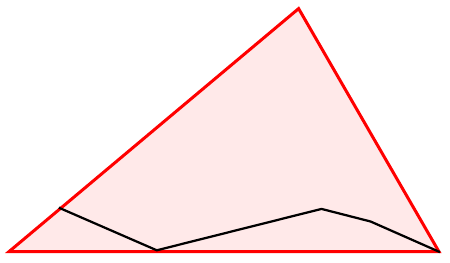}
\caption*{$40^{\circ}-60^{\circ}$}
\end{subfigure}

\caption{Continued results of triangle covering curve (Black curve is escape path, red curve is forest boundary, and caption is the base angle)}
\end{figure}

\begin{figure}[p]
\ContinuedFloat
\centering
\begin{subfigure}[b]{0.326\linewidth}
\centering
\includegraphics[width=1\linewidth]{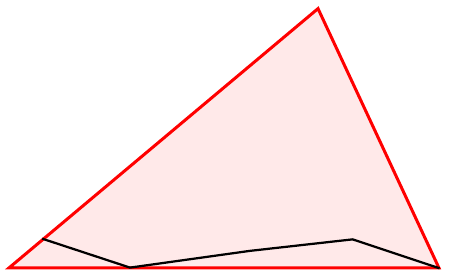}
\caption*{$40^{\circ}-65^{\circ}$}
\end{subfigure}
\begin{subfigure}[b]{0.326\linewidth}
\centering
\includegraphics[width=1\linewidth]{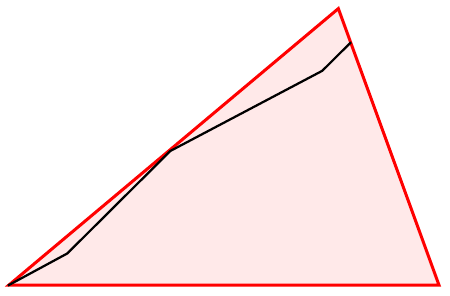}
\caption*{$40^{\circ}-70^{\circ}$}
\end{subfigure}
\begin{subfigure}[b]{0.326\linewidth}
\centering
\includegraphics[width=1\linewidth]{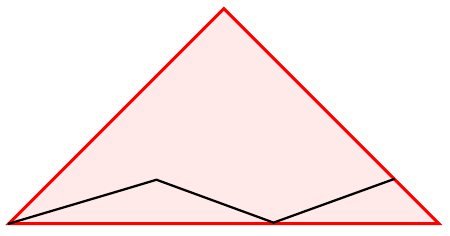}
\caption*{$45^{\circ}-45^{\circ}$}
\end{subfigure}
\begin{subfigure}[b]{0.326\linewidth}
\centering
\includegraphics[width=1\linewidth]{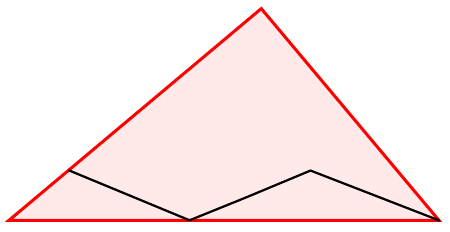}
\caption*{$45^{\circ}-50^{\circ}$}
\end{subfigure}
\begin{subfigure}[b]{0.326\linewidth}
\centering
\includegraphics[width=1\linewidth]{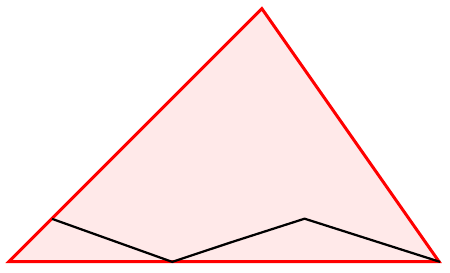}
\caption*{$45^{\circ}-55^{\circ}$}
\end{subfigure}
\begin{subfigure}[b]{0.326\linewidth}
\centering
\includegraphics[width=1\linewidth]{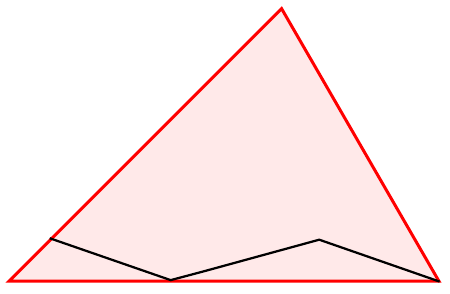}
\caption*{$45^{\circ}-60^{\circ}$}
\end{subfigure}
\begin{subfigure}[b]{0.326\linewidth}
\centering
\includegraphics[width=1\linewidth]{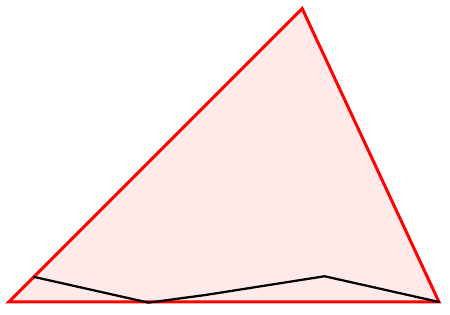}
\caption*{$45^{\circ}-65^{\circ}$}
\end{subfigure}
\begin{subfigure}[b]{0.326\linewidth}
\centering
\includegraphics[width=1\linewidth]{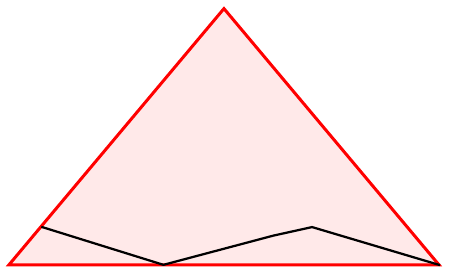}
\caption*{$50^{\circ}-50^{\circ}$}
\end{subfigure}
\begin{subfigure}[b]{0.326\linewidth}
\centering
\includegraphics[width=1\linewidth]{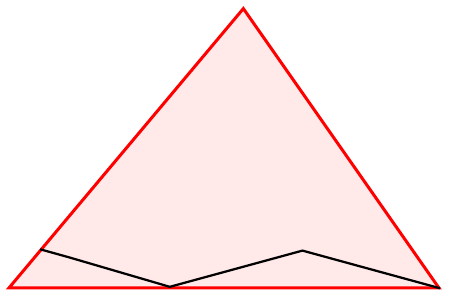}
\caption*{$50^{\circ}-55^{\circ}$}
\end{subfigure}
\begin{subfigure}[b]{0.326\linewidth}
\centering
\includegraphics[width=1\linewidth]{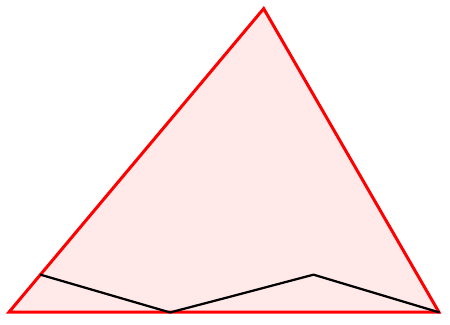}
\caption*{$50^{\circ}-60^{\circ}$}
\end{subfigure}
\begin{subfigure}[b]{0.326\linewidth}
\centering
\includegraphics[width=1\linewidth]{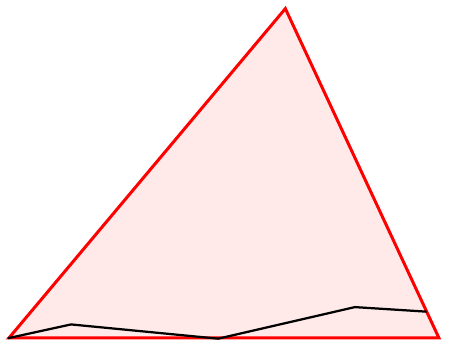}
\caption*{$50^{\circ}-65^{\circ}$}
\end{subfigure}
\begin{subfigure}[b]{0.326\linewidth}
\centering
\includegraphics[width=1\linewidth]{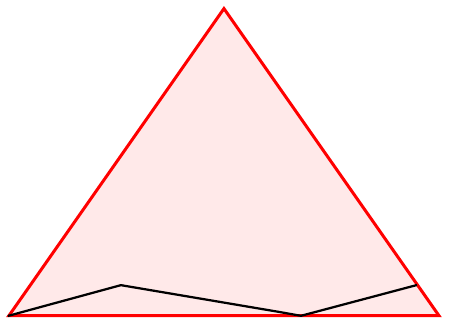}
\caption*{$55^{\circ}-55^{\circ}$}
\end{subfigure}
\begin{subfigure}[b]{0.326\linewidth}
\centering
\includegraphics[width=1\linewidth]{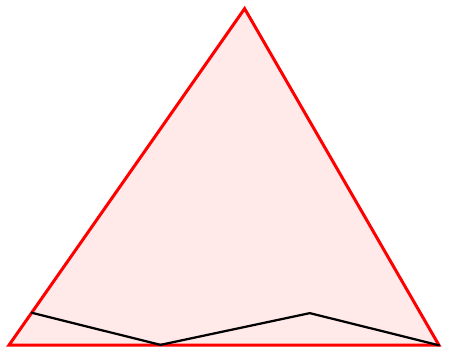}
\caption*{$55^{\circ}-60^{\circ}$}
\end{subfigure}
\begin{subfigure}[b]{0.326\linewidth}
\centering
\includegraphics[width=1\linewidth]{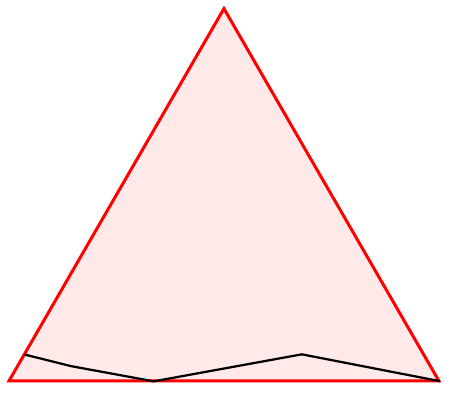}
\caption*{$60^{\circ}-60^{\circ}$}
\end{subfigure}
\caption{Continued results of triangle covering curve (Black curve is escape path, red curve is forest boundary, and caption is the base angle)}
\end{figure}

Using the duality between Bellman's lost-in-a-forest problem and Moser's worm problem \cite{Finch2004} \cite{Deng2024}, we can derive that, for arbitrary triangle, if the optimal curve length solution in Eq (8) is $L$ , the upper bound of Moser's worm problem is 
\begin{equation}
\frac{1}{2 L^2 \left(\frac{1}{\tan \alpha}+\frac{1}{\tan \beta }\right)}
\end{equation}

\begin{figure}
    \centering
    \includegraphics[width=0.56\linewidth]{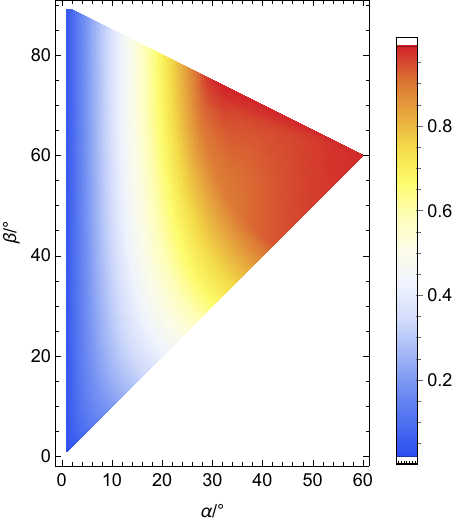}
    \includegraphics[width=0.56\linewidth]{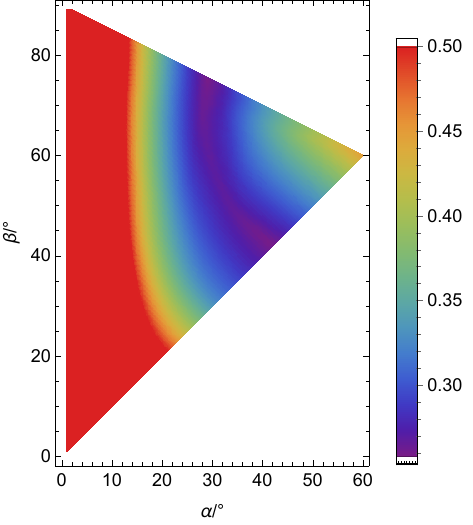}
    \caption{Density plot of escape path length of Bellman's forest (top) and area of Moser’s worm cover (bottom) in various base angles}
    \label{fig:placeholder}
\end{figure}

Figure 3 shows the density plot of escape path length of Bellman's triangle forest and area of Moser’s worm triangle cover calculated from Eq (22) in various base angles. We can observe a band exhibiting a noticeably smaller area for Moser's worm cover. Two isosceles triangles mentioned by Gibbs \cite{Gibbs2016} and 30-60 right triangle mentioned recently \cite{Wichiramala2026} are located in this area. Therefore, more refined numerical results and certificates for each angle can yield better upper bound for Moser's Worm Problem.

\subsection{Triangle covering closed polygonal curve}

Figure 4 shows the numerical results of triangle covering closed polygonal curve including triangle, quadrilateral, pentagon. We solved the optimization in Eq (13) with discretizing interval $[0, 2\pi]$ into 1440 parts. The figure shows the results when base angles are multiples of 10 degrees, and $0<\alpha \le \beta \le \pi - \alpha - \beta$, without duplications. The convergence criteria are the same as those in Section 3.1.

The results can provide insights into fit and cover of closed curves \cite{Brass2005} \cite {Wetzel2003} \cite {Füredi2011}, the general problem of universal covering a polygon with triangle or polygon, and fitting a polygon into triangle or polygon \cite{Wetzel2003}.

Future work of this paper involves using high-precision numerical methods to solve for each distinct angle for arbitrary triangles, and obtaining certificate of each optimized results.

\begin{figure}[p]
\centering
\begin{subfigure}[b]{0.326\linewidth}
\centering
\includegraphics[width=1\linewidth]{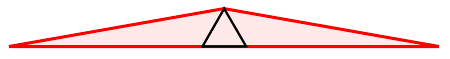}
\caption*{$10^{\circ}-10^{\circ}$}
\end{subfigure}
\begin{subfigure}[b]{0.326\linewidth}
\centering
\includegraphics[width=1\linewidth]{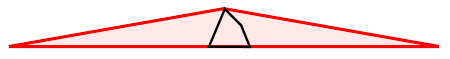}
\caption*{$10^{\circ}-10^{\circ}$}
\end{subfigure}
\begin{subfigure}[b]{0.326\linewidth}
\centering
\includegraphics[width=1\linewidth]{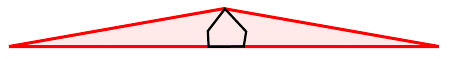}
\caption*{$10^{\circ}-10^{\circ}$}
\end{subfigure}
\begin{subfigure}[b]{0.326\linewidth}
\centering
\includegraphics[width=1\linewidth]{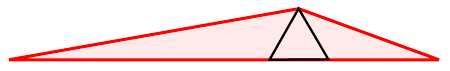}
\caption*{$10^{\circ}-20^{\circ}$}
\end{subfigure}
\begin{subfigure}[b]{0.326\linewidth}
\centering
\includegraphics[width=1\linewidth]{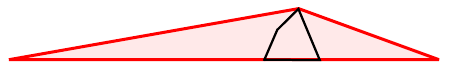}
\caption*{$10^{\circ}-20^{\circ}$}
\end{subfigure}
\begin{subfigure}[b]{0.326\linewidth}
\centering
\includegraphics[width=1\linewidth]{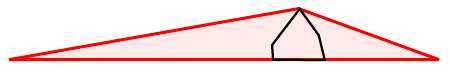}
\caption*{$10^{\circ}-20^{\circ}$}
\end{subfigure}
\begin{subfigure}[b]{0.326\linewidth}
\centering
\includegraphics[width=1\linewidth]{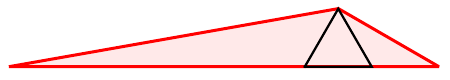}
\caption*{$10^{\circ}-30^{\circ}$}
\end{subfigure}
\begin{subfigure}[b]{0.326\linewidth}
\centering
\includegraphics[width=1\linewidth]{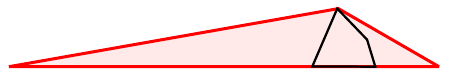}
\caption*{$10^{\circ}-30^{\circ}$}
\end{subfigure}
\begin{subfigure}[b]{0.326\linewidth}
\centering
\includegraphics[width=1\linewidth]{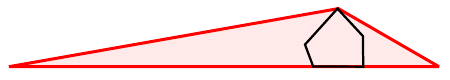}
\caption*{$10^{\circ}-30^{\circ}$}
\end{subfigure}
\begin{subfigure}[b]{0.326\linewidth}
\centering
\includegraphics[width=1\linewidth]{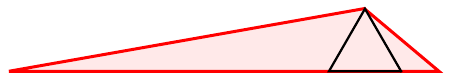}
\caption*{$10^{\circ}-40^{\circ}$}
\end{subfigure}
\begin{subfigure}[b]{0.326\linewidth}
\centering
\includegraphics[width=1\linewidth]{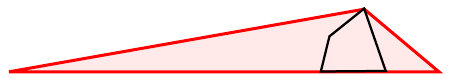}
\caption*{$10^{\circ}-40^{\circ}$}
\end{subfigure}
\begin{subfigure}[b]{0.326\linewidth}
\centering
\includegraphics[width=1\linewidth]{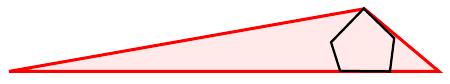}
\caption*{$10^{\circ}-40^{\circ}$}
\end{subfigure}
\begin{subfigure}[b]{0.326\linewidth}
\centering
\includegraphics[width=1\linewidth]{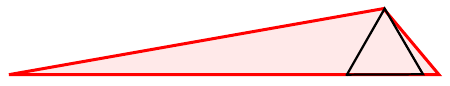}
\caption*{$10^{\circ}-50^{\circ}$}
\end{subfigure}
\begin{subfigure}[b]{0.326\linewidth}
\centering
\includegraphics[width=1\linewidth]{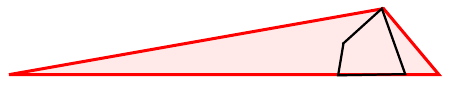}
\caption*{$10^{\circ}-50^{\circ}$}
\end{subfigure}
\begin{subfigure}[b]{0.326\linewidth}
\centering
\includegraphics[width=1\linewidth]{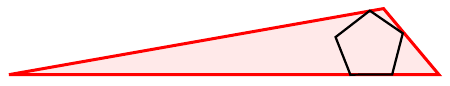}
\caption*{$10^{\circ}-50^{\circ}$}
\end{subfigure}
\begin{subfigure}[b]{0.326\linewidth}
\centering
\includegraphics[width=1\linewidth]{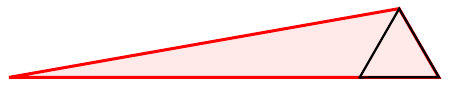}
\caption*{$10^{\circ}-60^{\circ}$}
\end{subfigure}
\begin{subfigure}[b]{0.326\linewidth}
\centering
\includegraphics[width=1\linewidth]{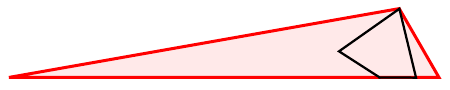}
\caption*{$10^{\circ}-60^{\circ}$}
\end{subfigure}
\begin{subfigure}[b]{0.326\linewidth}
\centering
\includegraphics[width=1\linewidth]{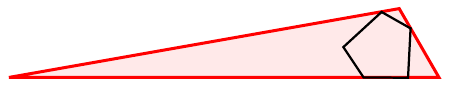}
\caption*{$10^{\circ}-60^{\circ}$}
\end{subfigure}
\begin{subfigure}[b]{0.326\linewidth}
\centering
\includegraphics[width=1\linewidth]{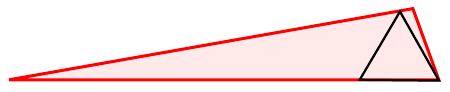}
\caption*{$10^{\circ}-70^{\circ}$}
\end{subfigure}
\begin{subfigure}[b]{0.326\linewidth}
\centering
\includegraphics[width=1\linewidth]{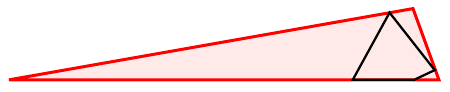}
\caption*{$10^{\circ}-70^{\circ}$}
\end{subfigure}
\begin{subfigure}[b]{0.326\linewidth}
\centering
\includegraphics[width=1\linewidth]{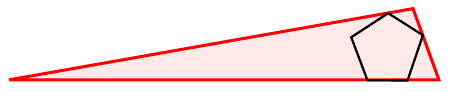}
\caption*{$10^{\circ}-70^{\circ}$}
\end{subfigure}
\begin{subfigure}[b]{0.326\linewidth}
\centering
\includegraphics[width=1\linewidth]{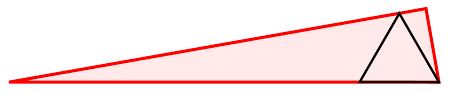}
\caption*{$10^{\circ}-80^{\circ}$}
\end{subfigure}
\begin{subfigure}[b]{0.326\linewidth}
\centering
\includegraphics[width=1\linewidth]{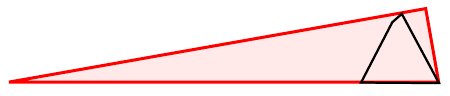}
\caption*{$10^{\circ}-80^{\circ}$}
\end{subfigure}
\begin{subfigure}[b]{0.326\linewidth}
\centering
\includegraphics[width=1\linewidth]{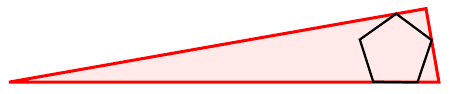}
\caption*{$10^{\circ}-80^{\circ}$}
\end{subfigure}
\begin{subfigure}[b]{0.326\linewidth}
\centering
\includegraphics[width=1\linewidth]{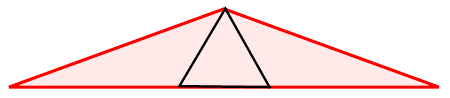}
\caption*{$20^{\circ}-20^{\circ}$}
\end{subfigure}
\begin{subfigure}[b]{0.326\linewidth}
\centering
\includegraphics[width=1\linewidth]{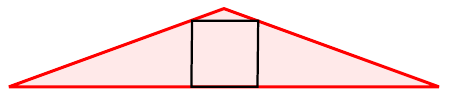}
\caption*{$20^{\circ}-20^{\circ}$}
\end{subfigure}
\begin{subfigure}[b]{0.326\linewidth}
\centering
\includegraphics[width=1\linewidth]{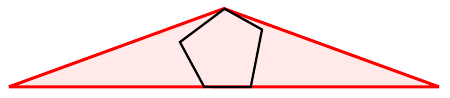}
\caption*{$20^{\circ}-20^{\circ}$}
\end{subfigure}
\begin{subfigure}[b]{0.326\linewidth}
\centering
\includegraphics[width=1\linewidth]{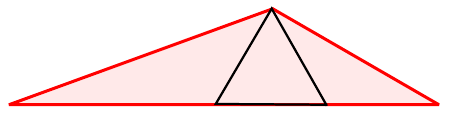}
\caption*{$20^{\circ}-30^{\circ}$}
\end{subfigure}
\begin{subfigure}[b]{0.326\linewidth}
\centering
\includegraphics[width=1\linewidth]{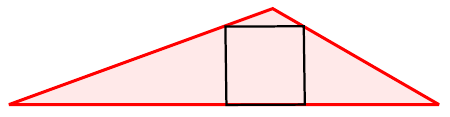}
\caption*{$20^{\circ}-30^{\circ}$}
\end{subfigure}
\begin{subfigure}[b]{0.326\linewidth}
\centering
\includegraphics[width=1\linewidth]{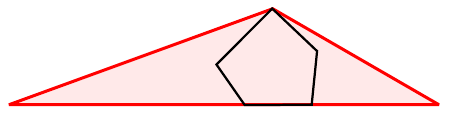}
\caption*{$20^{\circ}-30^{\circ}$}
\end{subfigure}
\begin{subfigure}[b]{0.326\linewidth}
\centering
\includegraphics[width=1\linewidth]{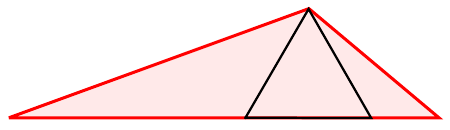}
\caption*{$20^{\circ}-40^{\circ}$}
\end{subfigure}
\begin{subfigure}[b]{0.326\linewidth}
\centering
\includegraphics[width=1\linewidth]{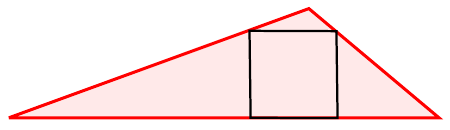}
\caption*{$20^{\circ}-40^{\circ}$}
\end{subfigure}
\begin{subfigure}[b]{0.326\linewidth}
\centering
\includegraphics[width=1\linewidth]{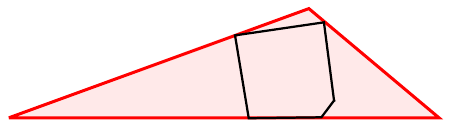}
\caption*{$20^{\circ}-40^{\circ}$}
\end{subfigure}
\caption{Results of triangle covering closed polygonal curve including triangle, quadrilateral, pentagon (Black curve is escape path, red curve is forest boundary, and caption is the base angle)}
\end{figure}

\begin{figure}[p]
\ContinuedFloat
\centering
\begin{subfigure}[b]{0.326\linewidth}
\centering
\includegraphics[width=1\linewidth]{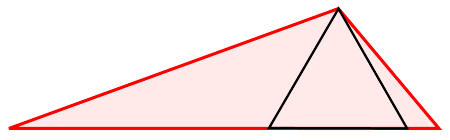}
\caption*{$20^{\circ}-50^{\circ}$}
\end{subfigure}
\begin{subfigure}[b]{0.326\linewidth}
\centering
\includegraphics[width=1\linewidth]{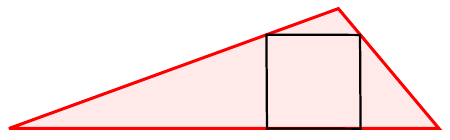}
\caption*{$20^{\circ}-50^{\circ}$}
\end{subfigure}
\begin{subfigure}[b]{0.326\linewidth}
\centering
\includegraphics[width=1\linewidth]{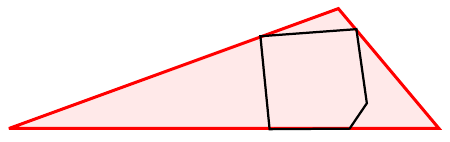}
\caption*{$20^{\circ}-50^{\circ}$}
\end{subfigure}
\begin{subfigure}[b]{0.326\linewidth}
\centering
\includegraphics[width=1\linewidth]{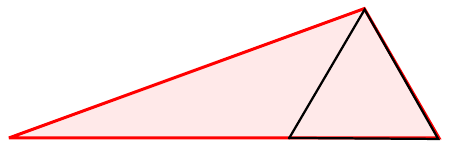}
\caption*{$20^{\circ}-60^{\circ}$}
\end{subfigure}
\begin{subfigure}[b]{0.326\linewidth}
\centering
\includegraphics[width=1\linewidth]{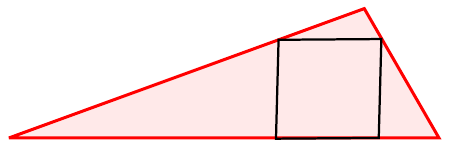}
\caption*{$20^{\circ}-60^{\circ}$}
\end{subfigure}
\begin{subfigure}[b]{0.326\linewidth}
\centering
\includegraphics[width=1\linewidth]{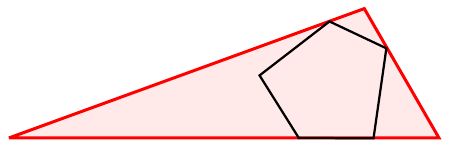}
\caption*{$20^{\circ}-60^{\circ}$}
\end{subfigure}
\begin{subfigure}[b]{0.326\linewidth}
\centering
\includegraphics[width=1\linewidth]{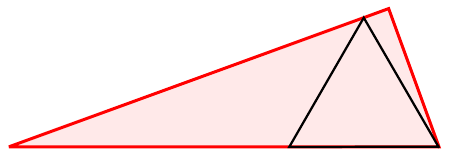}
\caption*{$20^{\circ}-70^{\circ}$}
\end{subfigure}
\begin{subfigure}[b]{0.326\linewidth}
\centering
\includegraphics[width=1\linewidth]{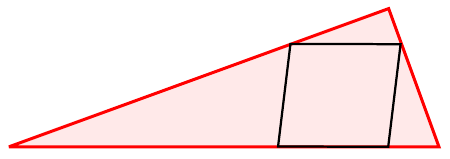}
\caption*{$20^{\circ}-70^{\circ}$}
\end{subfigure}
\begin{subfigure}[b]{0.326\linewidth}
\centering
\includegraphics[width=1\linewidth]{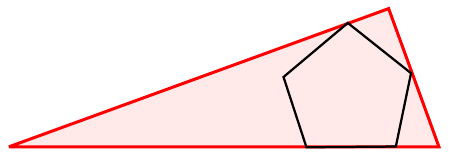}
\caption*{$20^{\circ}-70^{\circ}$}
\end{subfigure}
\begin{subfigure}[b]{0.326\linewidth}
\centering
\includegraphics[width=1\linewidth]{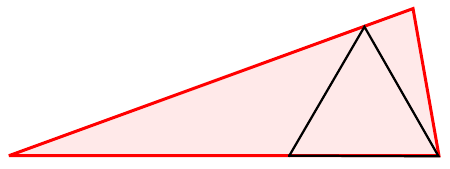}
\caption*{$20^{\circ}-80^{\circ}$}
\end{subfigure}
\begin{subfigure}[b]{0.326\linewidth}
\centering
\includegraphics[width=1\linewidth]{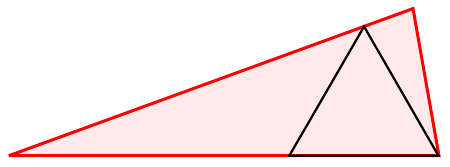}
\caption*{$20^{\circ}-80^{\circ}$}
\end{subfigure}
\begin{subfigure}[b]{0.326\linewidth}
\centering
\includegraphics[width=1\linewidth]{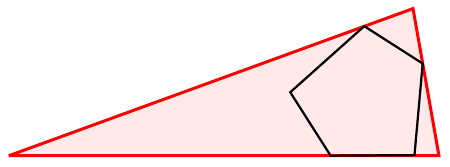}
\caption*{$20^{\circ}-80^{\circ}$}
\end{subfigure}
\begin{subfigure}[b]{0.326\linewidth}
\centering
\includegraphics[width=1\linewidth]{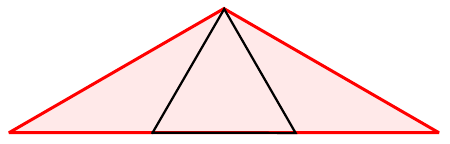}
\caption*{$30^{\circ}-30^{\circ}$}
\end{subfigure}
\begin{subfigure}[b]{0.326\linewidth}
\centering
\includegraphics[width=1\linewidth]{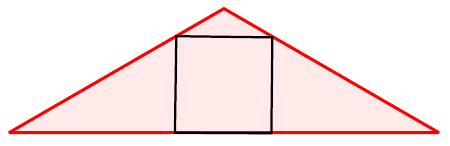}
\caption*{$30^{\circ}-30^{\circ}$}
\end{subfigure}
\begin{subfigure}[b]{0.326\linewidth}
\centering
\includegraphics[width=1\linewidth]{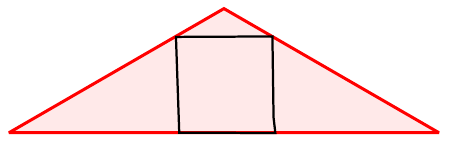}
\caption*{$30^{\circ}-30^{\circ}$}
\end{subfigure}
\begin{subfigure}[b]{0.326\linewidth}
\centering
\includegraphics[width=1\linewidth]{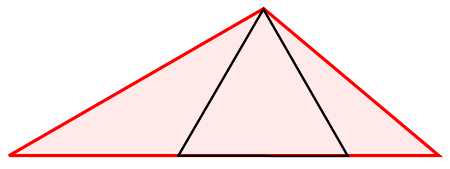}
\caption*{$30^{\circ}-40^{\circ}$}
\end{subfigure}
\begin{subfigure}[b]{0.326\linewidth}
\centering
\includegraphics[width=1\linewidth]{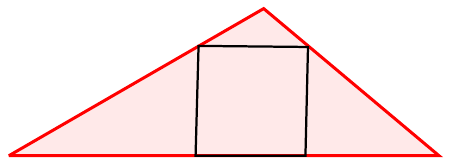}
\caption*{$30^{\circ}-40^{\circ}$}
\end{subfigure}
\begin{subfigure}[b]{0.326\linewidth}
\centering
\includegraphics[width=1\linewidth]{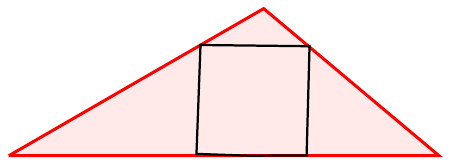}
\caption*{$30^{\circ}-40^{\circ}$}
\end{subfigure}
\begin{subfigure}[b]{0.326\linewidth}
\centering
\includegraphics[width=1\linewidth]{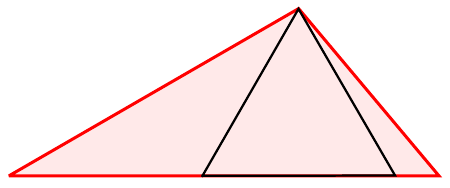}
\caption*{$30^{\circ}-50^{\circ}$}
\end{subfigure}
\begin{subfigure}[b]{0.326\linewidth}
\centering
\includegraphics[width=1\linewidth]{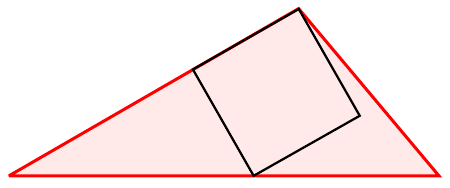}
\caption*{$30^{\circ}-50^{\circ}$}
\end{subfigure}
\begin{subfigure}[b]{0.326\linewidth}
\centering
\includegraphics[width=1\linewidth]{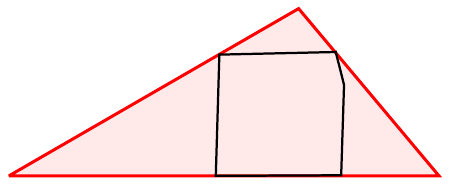}
\caption*{$30^{\circ}-50^{\circ}$}
\end{subfigure}
\caption{Continued results of triangle covering closed polygonal curve including triangle, quadrilateral, pentagon (Black curve is escape path, red curve is forest boundary, and caption is the base angle)}
\end{figure}

\begin{figure}[p]
\ContinuedFloat
\centering
\begin{subfigure}[b]{0.326\linewidth}
\centering
\includegraphics[width=1\linewidth]{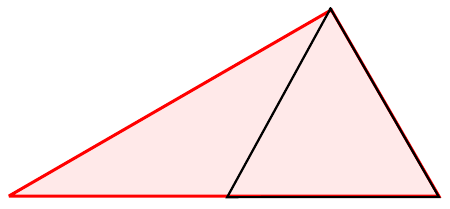}
\caption*{$30^{\circ}-60^{\circ}$}
\end{subfigure}
\begin{subfigure}[b]{0.326\linewidth}
\centering
\includegraphics[width=1\linewidth]{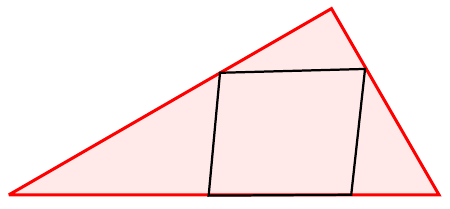}
\caption*{$30^{\circ}-60^{\circ}$}
\end{subfigure}
\begin{subfigure}[b]{0.326\linewidth}
\centering
\includegraphics[width=1\linewidth]{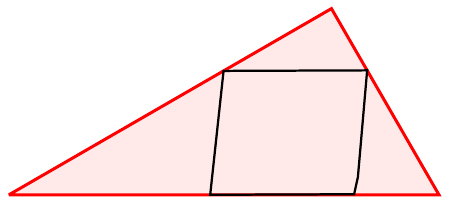}
\caption*{$30^{\circ}-60^{\circ}$}
\end{subfigure}
\begin{subfigure}[b]{0.326\linewidth}
\centering
\includegraphics[width=1\linewidth]{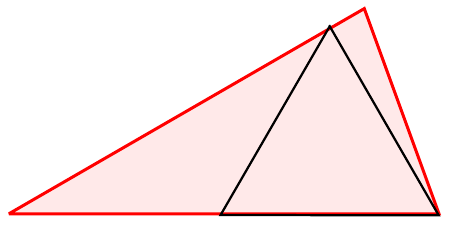}
\caption*{$30^{\circ}-70^{\circ}$}
\end{subfigure}
\begin{subfigure}[b]{0.326\linewidth}
\centering
\includegraphics[width=1\linewidth]{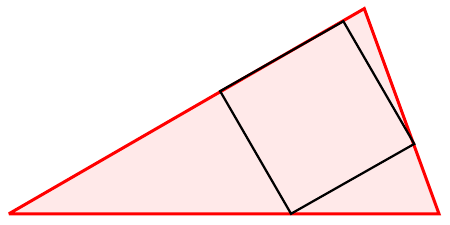}
\caption*{$30^{\circ}-70^{\circ}$}
\end{subfigure}
\begin{subfigure}[b]{0.326\linewidth}
\centering
\includegraphics[width=1\linewidth]{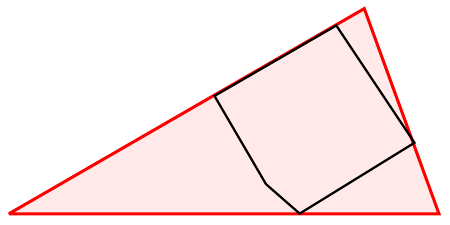}
\caption*{$30^{\circ}-70^{\circ}$}
\end{subfigure}
\begin{subfigure}[b]{0.326\linewidth}
\centering
\includegraphics[width=1\linewidth]{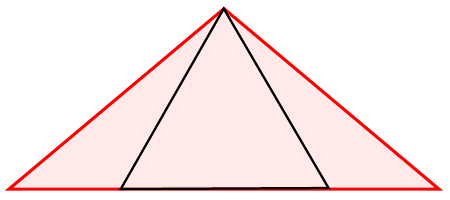}
\caption*{$40^{\circ}-40^{\circ}$}
\end{subfigure}
\begin{subfigure}[b]{0.326\linewidth}
\centering
\includegraphics[width=1\linewidth]{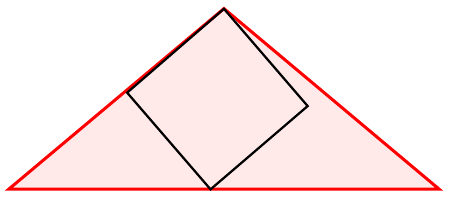}
\caption*{$40^{\circ}-40^{\circ}$}
\end{subfigure}
\begin{subfigure}[b]{0.326\linewidth}
\centering
\includegraphics[width=1\linewidth]{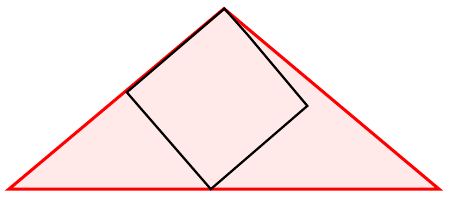}
\caption*{$40^{\circ}-40^{\circ}$}
\end{subfigure}
\begin{subfigure}[b]{0.326\linewidth}
\centering
\includegraphics[width=1\linewidth]{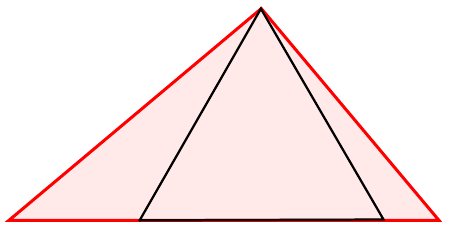}
\caption*{$40^{\circ}-50^{\circ}$}
\end{subfigure}
\begin{subfigure}[b]{0.326\linewidth}
\centering
\includegraphics[width=1\linewidth]{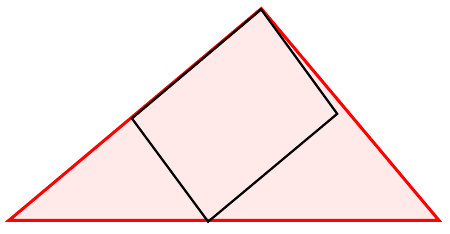}
\caption*{$40^{\circ}-50^{\circ}$}
\end{subfigure}
\begin{subfigure}[b]{0.326\linewidth}
\centering
\includegraphics[width=1\linewidth]{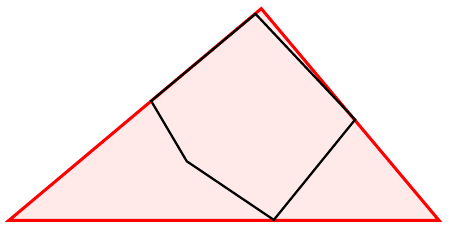}
\caption*{$40^{\circ}-50^{\circ}$}
\end{subfigure}
\begin{subfigure}[b]{0.326\linewidth}
\centering
\includegraphics[width=1\linewidth]{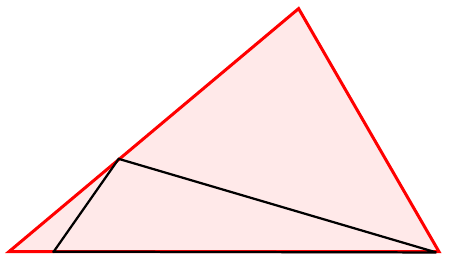}
\caption*{$40^{\circ}-60^{\circ}$}
\end{subfigure}
\begin{subfigure}[b]{0.326\linewidth}
\centering
\includegraphics[width=1\linewidth]{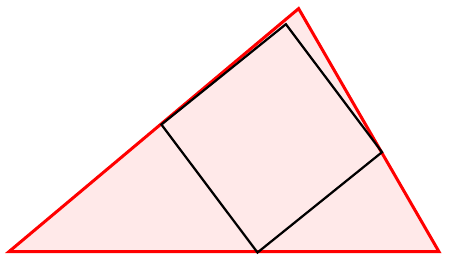}
\caption*{$40^{\circ}-60^{\circ}$}
\end{subfigure}
\begin{subfigure}[b]{0.326\linewidth}
\centering
\includegraphics[width=1\linewidth]{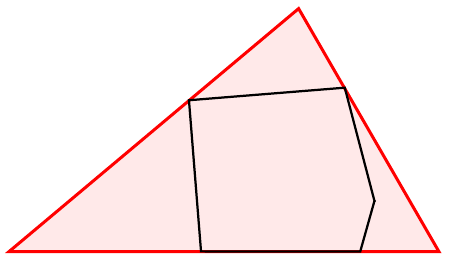}
\caption*{$40^{\circ}-60^{\circ}$}
\end{subfigure}
\caption{Continued results of triangle covering closed polygonal curve including triangle, quadrilateral, pentagon (Black curve is escape path, red curve is forest boundary, and caption is the base angle)}
\end{figure}

\begin{figure}[p]
\ContinuedFloat
\centering
\begin{subfigure}[b]{0.326\linewidth}
\centering
\includegraphics[width=1\linewidth]{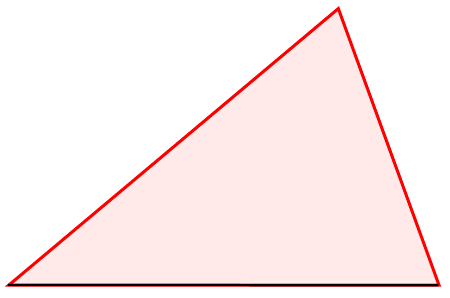}
\caption*{$40^{\circ}-70^{\circ}$}
\end{subfigure}
\begin{subfigure}[b]{0.326\linewidth}
\centering
\includegraphics[width=1\linewidth]{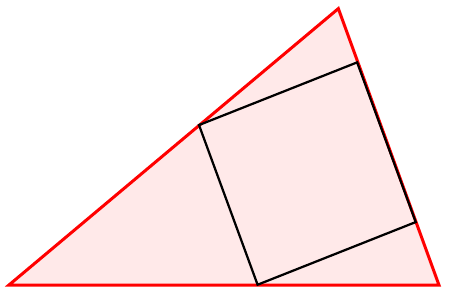}
\caption*{$40^{\circ}-70^{\circ}$}
\end{subfigure}
\begin{subfigure}[b]{0.326\linewidth}
\centering
\includegraphics[width=1\linewidth]{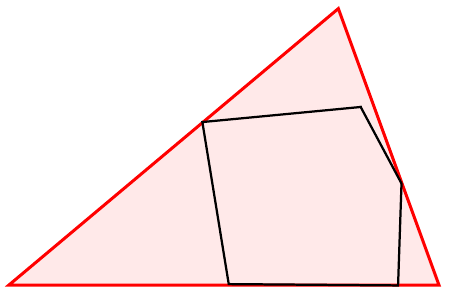}
\caption*{$40^{\circ}-70^{\circ}$}
\end{subfigure}
\begin{subfigure}[b]{0.326\linewidth}
\centering
\includegraphics[width=1\linewidth]{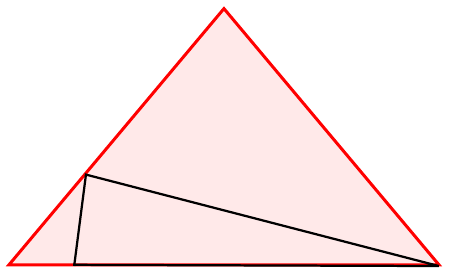}
\caption*{$50^{\circ}-50^{\circ}$}
\end{subfigure}
\begin{subfigure}[b]{0.326\linewidth}
\centering
\includegraphics[width=1\linewidth]{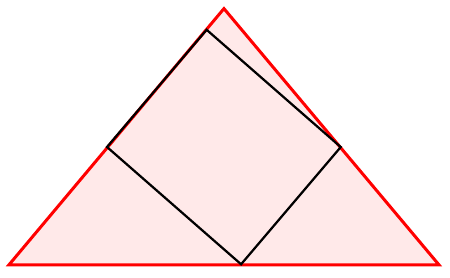}
\caption*{$50^{\circ}-50^{\circ}$}
\end{subfigure}
\begin{subfigure}[b]{0.326\linewidth}
\centering
\includegraphics[width=1\linewidth]{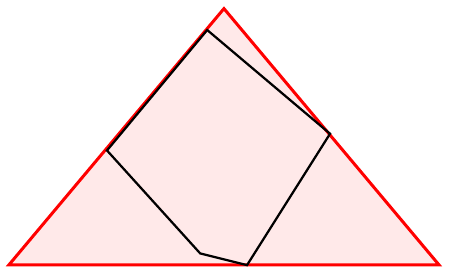}
\caption*{$50^{\circ}-50^{\circ}$}
\end{subfigure}
\begin{subfigure}[b]{0.326\linewidth}
\centering
\includegraphics[width=1\linewidth]{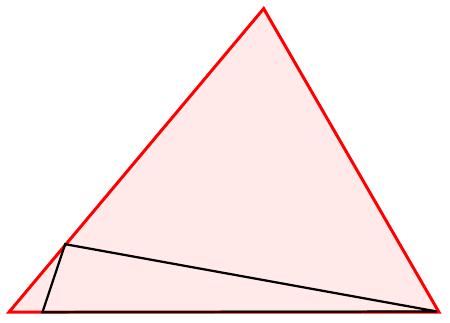}
\caption*{$50^{\circ}-60^{\circ}$}
\end{subfigure}
\begin{subfigure}[b]{0.326\linewidth}
\centering
\includegraphics[width=1\linewidth]{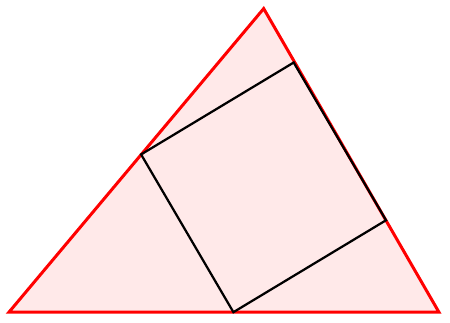}
\caption*{$50^{\circ}-60^{\circ}$}
\end{subfigure}
\begin{subfigure}[b]{0.326\linewidth}
\centering
\includegraphics[width=1\linewidth]{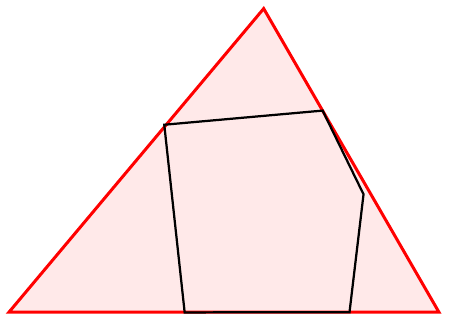}
\caption*{$50^{\circ}-60^{\circ}$}
\end{subfigure}
\begin{subfigure}[b]{0.326\linewidth}
\centering
\includegraphics[width=1\linewidth]{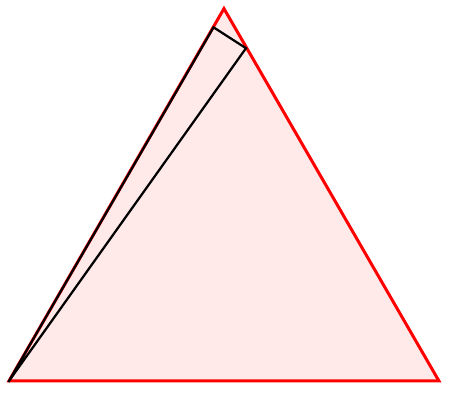}
\caption*{$60^{\circ}-60^{\circ}$}
\end{subfigure}
\begin{subfigure}[b]{0.326\linewidth}
\centering
\includegraphics[width=1\linewidth]{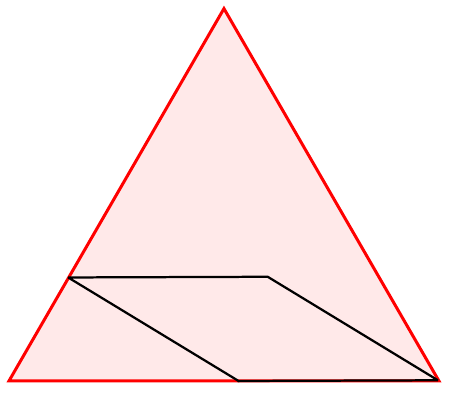}
\caption*{$60^{\circ}-60^{\circ}$}
\end{subfigure}
\begin{subfigure}[b]{0.326\linewidth}
\centering
\includegraphics[width=1\linewidth]{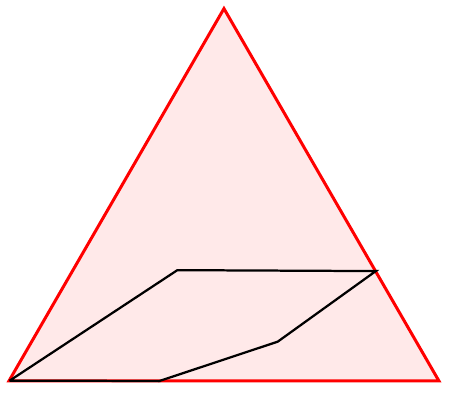}
\caption*{$60^{\circ}-60^{\circ}$}
\end{subfigure}
\caption{Continued results of triangle covering closed polygonal curve including triangle, quadrilateral, pentagon (Black curve is escape path, red curve is forest boundary, and caption is the base angle)}
\end{figure}

\section{Appendix-Formalized proofs in Lean}

The appendix provides formalized proofs of Theorems 4-8 in Lean 4 code.

The following is the Lean 4 code for Theorem 4 (Exact support function constraint characterization):
\begin{lstlisting}[style=LeanCode]
import Mathlib

noncomputable section

open Set

namespace TriangleEscape

/-!
# Exact support characterization for the triangular lost-in-a-forest problem

The formalization is split into four logically independent layers:

1. a trace/support lemma;
2. a weighted-simplex separation lemma;
3. the exact slack-coordinate description of the normalized triangle;
4. the all-starting-points and all-angles support theorem.
-/

section TraceSupport

variable {X : Type*}

/-- A number `H` is an attained support value of `p` on `Γ`. -/
def IsAttainedSupport (Γ : Set X) (p : X → ℝ) (H : ℝ) : Prop :=
  (∀ x ∈ Γ, p x ≤ H) ∧ ∃ x ∈ Γ, p x = H

/-- Strict membership in the intersection of three open half-spaces. -/
def StrictlyInside3
    (p₁ p₂ p₃ : X → ℝ) (d₁ d₂ d₃ : ℝ) (x : X) : Prop :=
  p₁ x < d₁ ∧ p₂ x < d₂ ∧ p₃ x < d₃

/-- The trace reaches the complement of the open three-half-space intersection. -/
def TraceEscapes3
    (Γ : Set X) (p₁ p₂ p₃ : X → ℝ) (d₁ d₂ d₃ : ℝ) : Prop :=
  ∃ x ∈ Γ, ¬ StrictlyInside3 p₁ p₂ p₃ d₁ d₂ d₃ x

/--
For attained support values, escape from the open intersection is equivalent to
one support reaching the corresponding offset.
-/
theorem traceEscapes3_iff_support
    {Γ : Set X} {p₁ p₂ p₃ : X → ℝ}
    {H₁ H₂ H₃ d₁ d₂ d₃ : ℝ}
    (hS₁ : IsAttainedSupport Γ p₁ H₁)
    (hS₂ : IsAttainedSupport Γ p₂ H₂)
    (hS₃ : IsAttainedSupport Γ p₃ H₃) :
    TraceEscapes3 Γ p₁ p₂ p₃ d₁ d₂ d₃ ↔
      d₁ ≤ H₁ ∨ d₂ ≤ H₂ ∨ d₃ ≤ H₃ := by
  constructor
  · rintro ⟨x, hxΓ, hxout⟩
    by_cases h₁ : d₁ ≤ p₁ x
    · exact Or.inl (h₁.trans (hS₁.1 x hxΓ))
    · by_cases h₂ : d₂ ≤ p₂ x
      · exact Or.inr (Or.inl (h₂.trans (hS₂.1 x hxΓ)))
      · by_cases h₃ : d₃ ≤ p₃ x
        · exact Or.inr (Or.inr (h₃.trans (hS₃.1 x hxΓ)))
        · exfalso
          apply hxout
          exact ⟨lt_of_not_ge h₁, lt_of_not_ge h₂, lt_of_not_ge h₃⟩
  · intro h
    rcases h with h₁ | h₂ | h₃
    · rcases hS₁.2 with ⟨x, hxΓ, hmax⟩
      refine ⟨x, hxΓ, ?_⟩
      intro hxinside
      have : H₁ < d₁ := by
        calc
          H₁ = p₁ x := hmax.symm
          _ < d₁ := hxinside.1
      exact (not_lt_of_ge h₁) this
    · rcases hS₂.2 with ⟨x, hxΓ, hmax⟩
      refine ⟨x, hxΓ, ?_⟩
      intro hxinside
      have : H₂ < d₂ := by
        calc
          H₂ = p₂ x := hmax.symm
          _ < d₂ := hxinside.2.1
      exact (not_lt_of_ge h₂) this
    · rcases hS₃.2 with ⟨x, hxΓ, hmax⟩
      refine ⟨x, hxΓ, ?_⟩
      intro hxinside
      have : H₃ < d₃ := by
        calc
          H₃ = p₃ x := hmax.symm
          _ < d₃ := hxinside.2.2
      exact (not_lt_of_ge h₃) this

/-- If the trace contains a point whose projection is zero, its support is nonnegative. -/
theorem support_nonneg_of_zero
    {Γ : Set X} {p : X → ℝ} {H : ℝ}
    (hS : IsAttainedSupport Γ p H)
    (hzero : ∃ x ∈ Γ, p x = 0) :
    0 ≤ H := by
  rcases hzero with ⟨x, hxΓ, hx0⟩
  simpa [hx0] using hS.1 x hxΓ

end TraceSupport

section WeightedSimplex

/-- Weighted sum of a triple. -/
def Weighted3
    (w₁ w₂ w₃ x₁ x₂ x₃ : ℝ) : ℝ :=
  w₁ * x₁ + w₂ * x₂ + w₃ * x₃

/-- The nonnegative weighted distance simplex. -/
def DistanceSimplex3
    (w₁ w₂ w₃ c d₁ d₂ d₃ : ℝ) : Prop :=
  0 ≤ d₁ ∧ 0 ≤ d₂ ∧ 0 ≤ d₃ ∧
    Weighted3 w₁ w₂ w₃ d₁ d₂ d₃ = c

/-- Every admissible distance triple is reached in at least one coordinate. -/
def RobustEscape3
    (w₁ w₂ w₃ c H₁ H₂ H₃ : ℝ) : Prop :=
  ∀ d₁ d₂ d₃,
    DistanceSimplex3 w₁ w₂ w₃ c d₁ d₂ d₃ →
      d₁ ≤ H₁ ∨ d₂ ≤ H₂ ∨ d₃ ≤ H₃

/--
The exact algebraic core: robust coordinatewise escape over a positive weighted
simplex is equivalent to one weighted support inequality.
-/
theorem robustEscape3_iff_weightedSupport
    {w₁ w₂ w₃ c H₁ H₂ H₃ : ℝ}
    (hw₁ : 0 < w₁) (hw₂ : 0 < w₂) (hw₃ : 0 < w₃)
    (hH₁ : 0 ≤ H₁) (hH₂ : 0 ≤ H₂) (hH₃ : 0 ≤ H₃) :
    RobustEscape3 w₁ w₂ w₃ c H₁ H₂ H₃ ↔
      c ≤ Weighted3 w₁ w₂ w₃ H₁ H₂ H₃ := by
  constructor
  · intro hrob
    by_contra hnot
    have hlt : Weighted3 w₁ w₂ w₃ H₁ H₂ H₃ < c :=
      lt_of_not_ge hnot
    let W : ℝ := w₁ + w₂ + w₃
    let δ : ℝ :=
      (c - Weighted3 w₁ w₂ w₃ H₁ H₂ H₃) / W
    have hW : 0 < W := by
      dsimp [W]
      linarith
    have hδ : 0 < δ := by
      dsimp [δ]
      exact div_pos (sub_pos.mpr hlt) hW
    have hδW : δ * W =
        c - Weighted3 w₁ w₂ w₃ H₁ H₂ H₃ := by
      dsimp [δ]
      field_simp [ne_of_gt hW]
    have hadm : DistanceSimplex3 w₁ w₂ w₃ c
        (H₁ + δ) (H₂ + δ) (H₃ + δ) := by
      refine ⟨by linarith, by linarith, by linarith, ?_⟩
      calc
        Weighted3 w₁ w₂ w₃ (H₁ + δ) (H₂ + δ) (H₃ + δ) =
            Weighted3 w₁ w₂ w₃ H₁ H₂ H₃ + δ * W := by
              dsimp [Weighted3, W]
              ring
        _ = Weighted3 w₁ w₂ w₃ H₁ H₂ H₃ +
              (c - Weighted3 w₁ w₂ w₃ H₁ H₂ H₃) := by rw [hδW]
        _ = c := by ring
    rcases hrob (H₁ + δ) (H₂ + δ) (H₃ + δ) hadm with h₁ | h₂ | h₃
    · linarith
    · linarith
    · linarith
  · intro hweighted d₁ d₂ d₃ hd
    rcases hd with ⟨hd₁, hd₂, hd₃, hsum⟩
    by_contra hnone
    push Not at hnone
    have h₁ : w₁ * H₁ < w₁ * d₁ :=
      mul_lt_mul_of_pos_left hnone.1 hw₁
    have h₂ : w₂ * H₂ < w₂ * d₂ :=
      mul_lt_mul_of_pos_left hnone.2.1 hw₂
    have h₃ : w₃ * H₃ < w₃ * d₃ :=
      mul_lt_mul_of_pos_left hnone.2.2 hw₃
    dsimp [Weighted3] at hweighted hsum
    linarith

end WeightedSimplex

section TriangleGeometry

abbrev Point := ℝ × ℝ

/-- Slack from the left supporting line. -/
def slack₁ (α : ℝ) (s : Point) : ℝ :=
  s.1 * Real.sin α - s.2 * Real.cos α

/-- Slack from the right supporting line. -/
def slack₂ (β : ℝ) (s : Point) : ℝ :=
  (1 - s.1) * Real.sin β - s.2 * Real.cos β

/-- Slack from the bottom supporting line. -/
def slack₃ (s : Point) : ℝ := s.2

/-- The normalized triangle, expressed entirely through nonnegative slacks. -/
def InTriangle (α β : ℝ) (s : Point) : Prop :=
  0 ≤ slack₁ α s ∧ 0 ≤ slack₂ β s ∧ 0 ≤ slack₃ s

/-- The weighted slack identity. -/
theorem weighted_slack_identity (α β : ℝ) (s : Point) :
    Real.sin β * slack₁ α s +
      Real.sin α * slack₂ β s +
      Real.sin (α + β) * slack₃ s
      = Real.sin α * Real.sin β := by
  rw [Real.sin_add]
  dsimp [slack₁, slack₂, slack₃]
  ring

/-- Explicit inverse of the first and third slack coordinates. -/
def pointOfSlacks (α d₁ d₃ : ℝ) : Point :=
  ((d₁ + d₃ * Real.cos α) / Real.sin α, d₃)

theorem slack₁_pointOfSlacks
    {α d₁ d₃ : ℝ} (hα : Real.sin α ≠ 0) :
    slack₁ α (pointOfSlacks α d₁ d₃) = d₁ := by
  dsimp [slack₁, pointOfSlacks]
  (field_simp [hα]; ring)

theorem slack₃_pointOfSlacks (α d₁ d₃ : ℝ) :
    slack₃ (pointOfSlacks α d₁ d₃) = d₃ := by
  rfl

theorem slack₂_pointOfSlacks
    {α β d₁ d₂ d₃ : ℝ}
    (hα : Real.sin α ≠ 0)
    (hrel : Real.sin β * d₁ + Real.sin α * d₂ +
      Real.sin (α + β) * d₃ = Real.sin α * Real.sin β) :
    slack₂ β (pointOfSlacks α d₁ d₃) = d₂ := by
  rw [Real.sin_add] at hrel
  dsimp [slack₂, pointOfSlacks]
  field_simp [hα]
  nlinarith [hrel]

/--
Every nonnegative triple satisfying the weighted identity is realized by a
unique triangle point (existence is all that is needed below).
-/
theorem distance_simplex_surjective
    {α β d₁ d₂ d₃ : ℝ}
    (hα : Real.sin α ≠ 0)
    (hd₁ : 0 ≤ d₁) (hd₂ : 0 ≤ d₂) (hd₃ : 0 ≤ d₃)
    (hrel : Real.sin β * d₁ + Real.sin α * d₂ +
      Real.sin (α + β) * d₃ = Real.sin α * Real.sin β) :
    ∃ s : Point,
      InTriangle α β s ∧
      slack₁ α s = d₁ ∧ slack₂ β s = d₂ ∧ slack₃ s = d₃ := by
  let s : Point := pointOfSlacks α d₁ d₃
  have h₁ : slack₁ α s = d₁ := by
    dsimp [s]
    exact slack₁_pointOfSlacks hα
  have h₂ : slack₂ β s = d₂ := by
    dsimp [s]
    exact slack₂_pointOfSlacks hα hrel
  have h₃ : slack₃ s = d₃ := by
    dsimp [s]
    exact slack₃_pointOfSlacks α d₁ d₃
  have hs : InTriangle α β s := by
    refine ⟨?_, ?_, ?_⟩
    · simpa [h₁] using hd₁
    · simpa [h₂] using hd₂
    · simpa [h₃] using hd₃
  exact ⟨s, hs, h₁, h₂, h₃⟩

/-- Positivity of the three triangle weights from the angle hypotheses. -/
theorem triangle_weight_positivity
    {α β : ℝ}
    (hα0 : 0 < α) (hβ0 : 0 < β)
    (hαβπ : α + β < Real.pi) :
    0 < Real.sin β ∧ 0 < Real.sin α ∧ 0 < Real.sin (α + β) := by
  have hαπ : α < Real.pi := by linarith
  have hβπ : β < Real.pi := by linarith
  exact ⟨
    Real.sin_pos_of_pos_of_lt_pi hβ0 hβπ,
    Real.sin_pos_of_pos_of_lt_pi hα0 hαπ,
    Real.sin_pos_of_pos_of_lt_pi (by linarith) hαβπ
  ⟩

end TriangleGeometry

section ExactTriangleTheorem

/-- Every starting point is escaped according to the three support values. -/
def CoversAllTriangleStarts
    (α β H₁ H₂ H₃ : ℝ) : Prop :=
  ∀ s : Point,
    InTriangle α β s →
      slack₁ α s ≤ H₁ ∨ slack₂ β s ≤ H₂ ∨ slack₃ s ≤ H₃

/--
Exact pointwise triangle theorem: coverage of all starting positions is
precisely the weighted support inequality.
-/
theorem coversAllTriangleStarts_iff_weightedSupport
    {α β H₁ H₂ H₃ : ℝ}
    (hα0 : 0 < α) (hβ0 : 0 < β)
    (hαβπ : α + β < Real.pi)
    (hH₁ : 0 ≤ H₁) (hH₂ : 0 ≤ H₂) (hH₃ : 0 ≤ H₃) :
    CoversAllTriangleStarts α β H₁ H₂ H₃ ↔
      Real.sin α * Real.sin β ≤
        Real.sin β * H₁ + Real.sin α * H₂ +
          Real.sin (α + β) * H₃ := by
  rcases triangle_weight_positivity hα0 hβ0 hαβπ with
    ⟨hw₁, hw₂, hw₃⟩
  have hαne : Real.sin α ≠ 0 := ne_of_gt hw₂
  have hstart_iff_distance :
      CoversAllTriangleStarts α β H₁ H₂ H₃ ↔
        RobustEscape3
          (Real.sin β) (Real.sin α) (Real.sin (α + β))
          (Real.sin α * Real.sin β) H₁ H₂ H₃ := by
    constructor
    · intro hstart d₁ d₂ d₃ hd
      rcases hd with ⟨hd₁, hd₂, hd₃, hrel⟩
      rcases distance_simplex_surjective hαne hd₁ hd₂ hd₃ hrel with
        ⟨s, hs, hs₁, hs₂, hs₃⟩
      rcases hstart s hs with h | h | h
      · exact Or.inl (by simpa [hs₁] using h)
      · exact Or.inr (Or.inl (by simpa [hs₂] using h))
      · exact Or.inr (Or.inr (by simpa [hs₃] using h))
    · intro hdist s hs
      apply hdist (slack₁ α s) (slack₂ β s) (slack₃ s)
      exact ⟨hs.1, hs.2.1, hs.2.2, weighted_slack_identity α β s⟩
  rw [hstart_iff_distance]
  simpa [Weighted3] using
    (robustEscape3_iff_weightedSupport hw₁ hw₂ hw₃ hH₁ hH₂ hH₃)

/-- The weighted support expression appearing in the paper. -/
def F (α β : ℝ) (h : ℝ → ℝ) (t : ℝ) : ℝ :=
  Real.sin β * h (t + Real.pi + α) +
    Real.sin α * h (t + Real.pi - β) +
    Real.sin (α + β) * h t

/-- Coverage of all starting points at the orientation phase `t`. -/
def CoversAtPhase (α β : ℝ) (h : ℝ → ℝ) (t : ℝ) : Prop :=
  CoversAllTriangleStarts α β
    (h (t + Real.pi + α))
    (h (t + Real.pi - β))
    (h t)

/--
The exact all-angle support-function characterization.

The theorem is stated on all real phases. Restricting `t` to `[0, 2π)` is
immediate when `h` is `2π`-periodic.
-/
theorem exact_support_function_characterization
    {α β : ℝ} {h : ℝ → ℝ}
    (hα0 : 0 < α) (hβ0 : 0 < β)
    (hαβπ : α + β < Real.pi)
    (h_nonneg : ∀ φ : ℝ, 0 ≤ h φ) :
    (∀ t : ℝ, CoversAtPhase α β h t) ↔
      ∀ t : ℝ, Real.sin α * Real.sin β ≤ F α β h t := by
  constructor
  · intro hcover t
    have ht :=
      (coversAllTriangleStarts_iff_weightedSupport
        hα0 hβ0 hαβπ
        (h_nonneg (t + Real.pi + α))
        (h_nonneg (t + Real.pi - β))
        (h_nonneg t)).mp (hcover t)
    simpa [CoversAtPhase, F] using ht
  · intro hF t
    apply
      (coversAllTriangleStarts_iff_weightedSupport
        hα0 hβ0 hαβπ
        (h_nonneg (t + Real.pi + α))
        (h_nonneg (t + Real.pi - β))
        (h_nonneg t)).mpr
    simpa [F] using hF t

end ExactTriangleTheorem

end TriangleEscape
\end{lstlisting}

The following is the Lean 4 code for Theorem 6 (Convergence of polygonal paths):
\begin{lstlisting}[style=LeanCode]
import Mathlib

open Filter
open scoped Topology

namespace TriangleCovering

noncomputable section

/-!
This file formalizes the algebraic and variational core of the polygonal
recovery argument for triangle-covering paths.

The geometric application uses
  L₁ = sin β / sin (α + β),
  L₂ = sin α / sin (α + β),
  L₃ = 1,
and C = sin α * sin β / sin (α + β).
After multiplication by sin (α + β), the corresponding support certificate
is the one in the paper.
-/


/-! ## Exact triangle identities -/

abbrev Vec2 := ℝ × ℝ

/-- The three side slacks in the normalized half-space representation. -/
def slack₁ (α : ℝ) (s : Vec2) : ℝ :=
  Real.sin α * s.1 - Real.cos α * s.2

def slack₂ (β : ℝ) (s : Vec2) : ℝ :=
  Real.sin β - Real.sin β * s.1 - Real.cos β * s.2

def slack₃ (s : Vec2) : ℝ := s.2

/-- The horizontal component of the edge-normal equilibrium identity. -/
theorem triangle_edge_equilibrium_x
    (α β : ℝ) (hs : Real.sin (α + β) ≠ 0) :
    (Real.sin β / Real.sin (α + β)) * (-Real.sin α) +
      (Real.sin α / Real.sin (α + β)) * Real.sin β = 0 := by
  field_simp [hs]
  ring

/-- The vertical component of the edge-normal equilibrium identity. -/
theorem triangle_edge_equilibrium_y
    (α β : ℝ) (hs : Real.sin (α + β) ≠ 0) :
    (Real.sin β / Real.sin (α + β)) * Real.cos α +
      (Real.sin α / Real.sin (α + β)) * Real.cos β - 1 = 0 := by
  field_simp [hs]
  rw [Real.sin_add]
  ring

/--
The weighted slack sum is independent of the starting point.  It equals twice
the area of the unit-base triangle.
-/
theorem triangle_weighted_slack_constant
    (α β : ℝ) (s : Vec2) (hs : Real.sin (α + β) ≠ 0) :
    (Real.sin β / Real.sin (α + β)) * slack₁ α s +
      (Real.sin α / Real.sin (α + β)) * slack₂ β s +
      slack₃ s =
        Real.sin α * Real.sin β / Real.sin (α + β) := by
  field_simp [hs]
  simp [slack₁, slack₂, slack₃, Real.sin_add]
  ring


/-- Membership in the triangle, expressed entirely by nonnegative side slacks. -/
def InTriangle (α β : ℝ) (s : Vec2) : Prop :=
  0 ≤ slack₁ α s ∧ 0 ≤ slack₂ β s ∧ 0 ≤ slack₃ s

/--
Every nonnegative point of the weighted slack simplex is generated by a point
of the triangle.  This is the missing surjectivity needed for necessity of the
single weighted support inequality.
-/
theorem triangle_slack_simplex_surjective
    {α β d₁ d₂ d₃ : ℝ}
    (hα : Real.sin α ≠ 0) (hαβ : Real.sin (α + β) ≠ 0)
    (hd₁ : 0 ≤ d₁) (hd₂ : 0 ≤ d₂) (hd₃ : 0 ≤ d₃)
    (hsum :
      (Real.sin β / Real.sin (α + β)) * d₁ +
        (Real.sin α / Real.sin (α + β)) * d₂ + d₃ =
          Real.sin α * Real.sin β / Real.sin (α + β)) :
    ∃ s : Vec2,
      InTriangle α β s ∧
      slack₁ α s = d₁ ∧ slack₂ β s = d₂ ∧ slack₃ s = d₃ := by
  let s : Vec2 :=
    ((d₁ + Real.cos α * d₃) / Real.sin α, d₃)
  have hs₁ : slack₁ α s = d₁ := by
    dsimp [s, slack₁]
    field_simp [hα]
    ring
  have hsum' := hsum
  field_simp [hαβ] at hsum'
  rw [Real.sin_add] at hsum'
  have hs₂ : slack₂ β s = d₂ := by
    dsimp [s, slack₂]
    field_simp [hα]
    nlinarith
  have hs₃ : slack₃ s = d₃ := by
    rfl
  refine ⟨s, ?_, hs₁, hs₂, hs₃⟩
  exact ⟨by simpa [hs₁] using hd₁, by simpa [hs₂] using hd₂,
    by simpa [hs₃] using hd₃⟩

/-- Robust domination of every nonnegative slack vector on a weighted simplex. -/
def RobustSlack
    (L₁ L₂ L₃ C h₁ h₂ h₃ : ℝ) : Prop :=
  ∀ d₁ d₂ d₃ : ℝ,
    0 ≤ d₁ → 0 ≤ d₂ → 0 ≤ d₃ →
    L₁ * d₁ + L₂ * d₂ + L₃ * d₃ = C →
    d₁ ≤ h₁ ∨ d₂ ≤ h₂ ∨ d₃ ≤ h₃

/--
The exact finite-dimensional certificate behind the triangle support formula:
for positive weights and nonnegative supports, robust domination of every
slack vector is equivalent to one weighted support inequality.
-/
theorem robustSlack_iff_weighted
    {L₁ L₂ L₃ C h₁ h₂ h₃ : ℝ}
    (hL₁ : 0 < L₁) (hL₂ : 0 < L₂) (hL₃ : 0 < L₃)
    (hh₁ : 0 ≤ h₁) (hh₂ : 0 ≤ h₂) (hh₃ : 0 ≤ h₃) :
    RobustSlack L₁ L₂ L₃ C h₁ h₂ h₃ ↔
      C ≤ L₁ * h₁ + L₂ * h₂ + L₃ * h₃ := by
  constructor
  · intro hrob
    let H : ℝ := L₁ * h₁ + L₂ * h₂ + L₃ * h₃
    let S : ℝ := L₁ + L₂ + L₃
    change C ≤ H
    by_contra hnot
    have hHC : H < C := lt_of_not_ge hnot
    have hS : 0 < S := by
      dsimp [S]
      linarith
    let e : ℝ := (C - H) / S
    have he : 0 < e := by
      dsimp [e]
      exact div_pos (sub_pos.mpr hHC) hS
    have heS : e * S = C - H := by
      dsimp [e]
      field_simp [ne_of_gt hS]
    have hsum :
        L₁ * (h₁ + e) + L₂ * (h₂ + e) + L₃ * (h₃ + e) = C := by
      calc
        L₁ * (h₁ + e) + L₂ * (h₂ + e) + L₃ * (h₃ + e)
            = H + e * S := by
                dsimp [H, S]
                ring
        _ = C := by linarith
    have hor := hrob (h₁ + e) (h₂ + e) (h₃ + e)
      (by linarith) (by linarith) (by linarith) hsum
    rcases hor with h | h | h <;> linarith
  · intro hweighted d₁ d₂ d₃ hd₁ hd₂ hd₃ hsum
    by_contra hnot
    have h₁lt : h₁ < d₁ := lt_of_not_ge (by
      intro h
      exact hnot (Or.inl h))
    have h₂lt : h₂ < d₂ := lt_of_not_ge (by
      intro h
      exact hnot (Or.inr (Or.inl h)))
    have h₃lt : h₃ < d₃ := lt_of_not_ge (by
      intro h
      exact hnot (Or.inr (Or.inr h)))
    have hm₁ : L₁ * h₁ < L₁ * d₁ := mul_lt_mul_of_pos_left h₁lt hL₁
    have hm₂ : L₂ * h₂ < L₂ * d₂ := mul_lt_mul_of_pos_left h₂lt hL₂
    have hm₃ : L₃ * h₃ < L₃ * d₃ := mul_lt_mul_of_pos_left h₃lt hL₃
    linarith


/-! ## Support as a supremum -/

/-- Supremal support of a scalar projection family. -/
noncomputable def supportSup {ι : Type*} (f : ι → ℝ) : ℝ :=
  sSup (Set.range f)

/-- A based path has nonnegative support because the zero projection is present. -/
theorem supportSup_nonneg
    {ι : Type*} {f : ι → ℝ} (i₀ : ι)
    (hf : BddAbove (Set.range f)) (hzero : f i₀ = 0) :
    0 ≤ supportSup f := by
  rw [← hzero]
  exact le_csSup hf ⟨i₀, rfl⟩

/--
Taking a supremum is 1-Lipschitz for the uniform norm.  This is the precise
order-theoretic core of support-function stability.
-/
theorem supportSup_error
    {ι : Type*} {f g : ι → ℝ} {ε : ℝ} (i₀ : ι)
    (hf : BddAbove (Set.range f)) (hg : BddAbove (Set.range g))
    (herr : ∀ i, |f i - g i| ≤ ε) :
    |supportSup f - supportSup g| ≤ ε := by
  have hfg : supportSup f ≤ supportSup g + ε := by
    apply csSup_le
    · exact ⟨f i₀, ⟨i₀, rfl⟩⟩
    · intro b hb
      rcases hb with ⟨i, rfl⟩
      have hgi : g i ≤ supportSup g := le_csSup hg ⟨i, rfl⟩
      have hdiff : f i ≤ g i + ε := by
        have haux := (abs_le.mp (herr i)).2
        linarith
      linarith
  have hgf : supportSup g ≤ supportSup f + ε := by
    apply csSup_le
    · exact ⟨g i₀, ⟨i₀, rfl⟩⟩
    · intro b hb
      rcases hb with ⟨i, rfl⟩
      have hfi : f i ≤ supportSup f := le_csSup hf ⟨i, rfl⟩
      have hdiff : g i ≤ f i + ε := by
        have haux := (abs_le.mp (herr i)).1
        linarith
      linarith
  rw [abs_le]
  constructor <;> linarith

/-- Three-term weighted support functional. -/
def weightedCover
    (w₁ w₂ w₃ : ℝ) (h : ℝ → ℝ) (φ₁ φ₂ φ₃ : ℝ) : ℝ :=
  w₁ * h φ₁ + w₂ * h φ₂ + w₃ * h φ₃

/--
A pointwise support error of at most `ε` produces a weighted-cover deficit of
at most `(w₁ + w₂ + w₃) ε`.
-/
theorem weightedCover_lower_of_pointwise_error
    {w₁ w₂ w₃ c ε : ℝ} {h g : ℝ → ℝ} {φ₁ φ₂ φ₃ : ℝ}
    (hw₁ : 0 ≤ w₁) (hw₂ : 0 ≤ w₂) (hw₃ : 0 ≤ w₃)
    (herr : ∀ φ, |h φ - g φ| ≤ ε)
    (hg : c ≤ weightedCover w₁ w₂ w₃ g φ₁ φ₂ φ₃) :
    c - (w₁ + w₂ + w₃) * ε ≤
      weightedCover w₁ w₂ w₃ h φ₁ φ₂ φ₃ := by
  have hφ₁ : g φ₁ - ε ≤ h φ₁ := by
    have haux := (abs_le.mp (herr φ₁)).1
    linarith
  have hφ₂ : g φ₂ - ε ≤ h φ₂ := by
    have haux := (abs_le.mp (herr φ₂)).1
    linarith
  have hφ₃ : g φ₃ - ε ≤ h φ₃ := by
    have haux := (abs_le.mp (herr φ₃)).1
    linarith
  have hm₁ : w₁ * (g φ₁ - ε) ≤ w₁ * h φ₁ :=
    mul_le_mul_of_nonneg_left hφ₁ hw₁
  have hm₂ : w₂ * (g φ₂ - ε) ≤ w₂ * h φ₂ :=
    mul_le_mul_of_nonneg_left hφ₂ hw₂
  have hm₃ : w₃ * (g φ₃ - ε) ≤ w₃ * h φ₃ :=
    mul_le_mul_of_nonneg_left hφ₃ hw₃
  calc
    c - (w₁ + w₂ + w₃) * ε
        ≤ weightedCover w₁ w₂ w₃ g φ₁ φ₂ φ₃ -
            (w₁ + w₂ + w₃) * ε := sub_le_sub_right hg _
    _ = w₁ * (g φ₁ - ε) + w₂ * (g φ₂ - ε) +
          w₃ * (g φ₃ - ε) := by
          simp [weightedCover]
          ring
    _ ≤ w₁ * h φ₁ + w₂ * h φ₂ + w₃ * h φ₃ :=
          add_le_add (add_le_add hm₁ hm₂) hm₃
    _ = weightedCover w₁ w₂ w₃ h φ₁ φ₂ φ₃ := rfl

/-- Positive homogeneity of the weighted support functional. -/
theorem weightedCover_scale
    (w₁ w₂ w₃ ρ : ℝ) (h : ℝ → ℝ) (φ₁ φ₂ φ₃ : ℝ) :
    weightedCover w₁ w₂ w₃ (fun φ => ρ * h φ) φ₁ φ₂ φ₃ =
      ρ * weightedCover w₁ w₂ w₃ h φ₁ φ₂ φ₃ := by
  simp [weightedCover]
  ring

/-- Scaling by `c / (c - δ)` repairs a uniform deficit `δ < c`. -/
theorem scale_repairs_deficit
    {c δ z : ℝ} (hc : 0 < c) (hδc : δ < c) (hz : c - δ ≤ z) :
    c ≤ (c / (c - δ)) * z := by
  have hden : 0 < c - δ := sub_pos.mpr hδc
  have hρ : 0 ≤ c / (c - δ) := (div_pos hc hden).le
  calc
    c = (c / (c - δ)) * (c - δ) := by
      field_simp [ne_of_gt hden]
    _ ≤ (c / (c - δ)) * z :=
      mul_le_mul_of_nonneg_left hz hρ

/--
The scalar recovery lemma used in the polygonal approximation proof.
It combines the support-error estimate, homogeneity, and the repair factor.
-/
theorem scaled_weighted_cover_feasible
    {w₁ w₂ w₃ c ε : ℝ} {h g : ℝ → ℝ} {φ₁ φ₂ φ₃ : ℝ}
    (hw₁ : 0 ≤ w₁) (hw₂ : 0 ≤ w₂) (hw₃ : 0 ≤ w₃)
    (hc : 0 < c)
    (hsmall : (w₁ + w₂ + w₃) * ε < c)
    (herr : ∀ φ, |h φ - g φ| ≤ ε)
    (hg : c ≤ weightedCover w₁ w₂ w₃ g φ₁ φ₂ φ₃) :
    c ≤ weightedCover w₁ w₂ w₃
      (fun φ => (c / (c - (w₁ + w₂ + w₃) * ε)) * h φ)
      φ₁ φ₂ φ₃ := by
  let δ : ℝ := (w₁ + w₂ + w₃) * ε
  have happrox : c - δ ≤ weightedCover w₁ w₂ w₃ h φ₁ φ₂ φ₃ := by
    dsimp [δ]
    exact weightedCover_lower_of_pointwise_error hw₁ hw₂ hw₃ herr hg
  have hscaled :
      c ≤ (c / (c - δ)) * weightedCover w₁ w₂ w₃ h φ₁ φ₂ φ₃ :=
    scale_repairs_deficit hc (by simpa [δ] using hsmall) happrox
  rw [weightedCover_scale]
  simpa [δ] using hscaled

/--
Abstract squeeze theorem for optimal values.  In the geometric application,
`v ≤ vK K` comes from inclusion of polygonal paths in all rectifiable paths,
and `vK K ≤ ρ K * v` comes from the scaled polygonal recovery sequence.
-/
theorem optimal_values_tendsto
    {v : ℝ} {vK ρ : ℕ → ℝ}
    (hlower : ∀ K, v ≤ vK K)
    (hupper : ∀ᶠ K in atTop, vK K ≤ ρ K * v)
    (hρ : Tendsto ρ atTop (𝓝 1)) :
    Tendsto vK atTop (𝓝 v) := by
  have hupperT : Tendsto (fun K => ρ K * v) atTop (𝓝 v) := by
    simpa using hρ.mul_const v
  exact tendsto_of_tendsto_of_tendsto_of_le_of_le'
    tendsto_const_nhds hupperT (Eventually.of_forall hlower) hupper

/--
Abstract compactness/lower-semicontinuity endpoint.  The geometric application
instantiates `Kset` with a uniformly bounded, equi-Lipschitz family after
constant-speed reparameterization.
-/
theorem compact_subsequence_is_optimal
    {Path : Type*} [TopologicalSpace Path] [FirstCountableTopology Path]
    {A Kset : Set Path} {cost : Path → ℝ} {v : ℝ} {p : ℕ → Path}
    (hcompact : IsCompact Kset)
    (hpK : ∀ n, p n ∈ Kset)
    (hfeasible_limit :
      ∀ (x : Path) (φ : ℕ → ℕ),
        StrictMono φ → Tendsto (p ∘ φ) atTop (𝓝 x) → x ∈ A)
    (hlsc_limit :
      ∀ (x : Path) (φ : ℕ → ℕ),
        StrictMono φ → Tendsto (p ∘ φ) atTop (𝓝 x) → cost x ≤ v)
    (hlower : ∀ x ∈ A, v ≤ cost x) :
    ∃ x ∈ A, cost x = v ∧
      ∃ φ : ℕ → ℕ, StrictMono φ ∧ Tendsto (p ∘ φ) atTop (𝓝 x) := by
  obtain ⟨x, hxK, φ, hφ, hconv⟩ :=
    hcompact.tendsto_subseq (x := p) hpK
  have hxA : x ∈ A := hfeasible_limit x φ hφ hconv
  have hcost_le : cost x ≤ v := hlsc_limit x φ hφ hconv
  have hcost_eq : cost x = v := le_antisymm hcost_le (hlower x hxA)
  exact ⟨x, hxA, hcost_eq, φ, hφ, hconv⟩

/-- Mathlib's lower semicontinuity theorem for total variation under uniform convergence. -/
theorem variation_lowerSemicontinuous_uniform
    {ι E : Type*} [LinearOrder ι] [PseudoEMetricSpace E] (s : Set ι) :
    LowerSemicontinuous
      (fun f : UniformOnFun ι E {s} => eVariationOn f s) :=
  eVariationOn.lowerSemicontinuous_uniformOn s

end

end TriangleCovering
\end{lstlisting}

The following is the Lean 4 code for Lemma 7 and Theorem 8 (Convergence of the fully discrete formulation):
\begin{lstlisting}[style=LeanCode]
import Mathlib

/-!
# Triangle covering curves: formal core

This file is `sorry`-free and introduces no axioms. It formalizes:

1. the scaled equilibrium identity for the three triangle normals;
2. the weighted side-slack identity;
3. the weighted certificate for all nonnegative slack triples;
4. the intermediate-value boundary-crossing lemma;
5. stability of an attained support maximum;
6. the weighted angular Lipschitz estimate;
7. the collocation-error estimate;
8. repair of an approximate support constraint by dilation;
9. fixed-complexity and joint convergence by quantitative squeezing.
-/

noncomputable section

open Filter Set

namespace TriangleCovering

abbrev Vec2 := ℝ × ℝ

def dot (u v : Vec2) : ℝ :=
  u.1 * v.1 + u.2 * v.2

def n₁ (α : ℝ) : Vec2 :=
  (-Real.sin α, Real.cos α)

def n₂ (β : ℝ) : Vec2 :=
  (Real.sin β, Real.cos β)

def n₃ : Vec2 :=
  (0, -1)

def slack₁ (α : ℝ) (s : Vec2) : ℝ :=
  -dot (n₁ α) s

def slack₂ (β : ℝ) (s : Vec2) : ℝ :=
  Real.sin β - dot (n₂ β) s

def slack₃ (s : Vec2) : ℝ :=
  -dot n₃ s

/--
The denominator-free equilibrium relation for the three outward normals.
-/
theorem scaled_normal_equilibrium (α β : ℝ) :
    Real.sin β • n₁ α +
        Real.sin α • n₂ β +
        Real.sin (α + β) • n₃ =
      (0 : Vec2) := by
  ext <;>
    simp [n₁, n₂, n₃, Real.sin_add] <;>
    ring

/--
The weighted sum of the three side slacks is independent of the point.
-/
theorem weighted_slack_identity (α β : ℝ) (s : Vec2) :
    Real.sin β * slack₁ α s +
          Real.sin α * slack₂ β s +
          Real.sin (α + β) * slack₃ s =
        Real.sin α * Real.sin β := by
  rcases s with ⟨x, y⟩
  simp [
    slack₁,
    slack₂,
    slack₃,
    dot,
    n₁,
    n₂,
    n₃,
    Real.sin_add
  ]
  ring

/--
A continuous scalar projection that starts below a level and later reaches
or exceeds it must attain that level.
-/
theorem reaches_level_of_starts_below
    {f : ℝ → ℝ} {d : ℝ}
    (hf : ContinuousOn f (Icc (0 : ℝ) 1))
    (h0 : f 0 ≤ d)
    (hp : ∃ p ∈ Icc (0 : ℝ) 1, d ≤ f p) :
    ∃ q ∈ Icc (0 : ℝ) 1, f q = d := by
  rcases hp with ⟨p, hp, hdp⟩
  have hsub :
      Icc (0 : ℝ) p ⊆ Icc (0 : ℝ) 1 := by
    intro x hx
    exact ⟨hx.1, hx.2.trans hp.2⟩
  have hfp :
      ContinuousOn f (Icc (0 : ℝ) p) :=
    hf.mono hsub
  have hdmem :
      d ∈ Icc (f 0) (f p) :=
    ⟨h0, hdp⟩
  have himg :
      d ∈ f '' Icc (0 : ℝ) p :=
    (intermediate_value_Icc hp.1 hfp) hdmem
  rcases himg with ⟨q, hq, hfq⟩
  exact ⟨q, hsub hq, hfq⟩

/--
Robust coverage of all nonnegative slack triples on a weighted affine
simplex.
-/
def CoversAllSlacks
    (w₁ w₂ w₃ c h₁ h₂ h₃ : ℝ) : Prop :=
  ∀ d₁ d₂ d₃ : ℝ,
    0 ≤ d₁ →
    0 ≤ d₂ →
    0 ≤ d₃ →
    w₁ * d₁ + w₂ * d₂ + w₃ * d₃ = c →
    d₁ ≤ h₁ ∨ d₂ ≤ h₂ ∨ d₃ ≤ h₃

/--
The weighted support inequality is equivalent to coverage of every
nonnegative slack triple with the prescribed weighted sum.
-/
theorem weighted_certificate_iff
    {w₁ w₂ w₃ c h₁ h₂ h₃ : ℝ}
    (hw₁ : 0 < w₁)
    (hw₂ : 0 < w₂)
    (hw₃ : 0 < w₃)
    (hh₁ : 0 ≤ h₁)
    (hh₂ : 0 ≤ h₂)
    (hh₃ : 0 ≤ h₃) :
    c ≤ w₁ * h₁ + w₂ * h₂ + w₃ * h₃ ↔
      CoversAllSlacks w₁ w₂ w₃ c h₁ h₂ h₃ := by
  constructor
  · intro hcert d₁ d₂ d₃ hd₁ hd₂ hd₃ hsum
    by_contra hnone
    have hn₁ : ¬d₁ ≤ h₁ := by
      intro h
      exact hnone (Or.inl h)
    have hn₂ : ¬d₂ ≤ h₂ := by
      intro h
      exact hnone (Or.inr (Or.inl h))
    have hn₃ : ¬d₃ ≤ h₃ := by
      intro h
      exact hnone (Or.inr (Or.inr h))
    have h₁d₁ : h₁ < d₁ :=
      lt_of_not_ge hn₁
    have h₂d₂ : h₂ < d₂ :=
      lt_of_not_ge hn₂
    have h₃d₃ : h₃ < d₃ :=
      lt_of_not_ge hn₃
    have hmul₁ :
        w₁ * h₁ < w₁ * d₁ :=
      mul_lt_mul_of_pos_left h₁d₁ hw₁
    have hmul₂ :
        w₂ * h₂ < w₂ * d₂ :=
      mul_lt_mul_of_pos_left h₂d₂ hw₂
    have hmul₃ :
        w₃ * h₃ < w₃ * d₃ :=
      mul_lt_mul_of_pos_left h₃d₃ hw₃
    linarith
  · intro hrobust
    by_contra hnot
    let H : ℝ :=
      w₁ * h₁ + w₂ * h₂ + w₃ * h₃
    have hltH : H < c := by
      dsimp [H]
      exact lt_of_not_ge hnot
    let W : ℝ :=
      w₁ + w₂ + w₃
    have hW : 0 < W := by
      dsimp [W]
      linarith
    let e : ℝ :=
      (c - H) / W
    have he : 0 < e := by
      dsimp [e]
      exact div_pos (sub_pos.mpr hltH) hW
    have heW :
        e * W = c - H := by
      dsimp [e]
      exact div_mul_cancel₀ (c - H) (ne_of_gt hW)
    have hd₁ :
        0 ≤ h₁ + e := by
      linarith
    have hd₂ :
        0 ≤ h₂ + e := by
      linarith
    have hd₃ :
        0 ≤ h₃ + e := by
      linarith
    have hsum :
        w₁ * (h₁ + e) +
              w₂ * (h₂ + e) +
              w₃ * (h₃ + e) =
            c := by
      calc
        w₁ * (h₁ + e) +
              w₂ * (h₂ + e) +
              w₃ * (h₃ + e) =
            H + e * W := by
              dsimp [H, W]
              ring
        _ = H + (c - H) := by
              rw [heW]
        _ = c := by
              ring
    rcases
        hrobust
          (h₁ + e)
          (h₂ + e)
          (h₃ + e)
          hd₁
          hd₂
          hd₃
          hsum
      with h | h | h <;>
      linarith

/--
If a pointwise family has an attained maximum and every member is
`B`-Lipschitz, then the maximum is `B`-Lipschitz.
-/
theorem attained_support_lipschitz
    {ι Θ : Type*}
    [PseudoMetricSpace Θ]
    (proj : ι → Θ → ℝ)
    (h : Θ → ℝ)
    (B : ℝ)
    (hupper : ∀ i t, proj i t ≤ h t)
    (hattain : ∀ t, ∃ i, h t = proj i t)
    (hproj :
      ∀ i t s,
        |proj i t - proj i s| ≤ B * dist t s)
    (t s : Θ) :
    |h t - h s| ≤ B * dist t s := by
  rcases hattain t with ⟨i, hi⟩
  rcases hattain s with ⟨j, hj⟩
  have hts :
      h t - h s ≤ B * dist t s := by
    calc
      h t - h s =
          proj i t - h s := by
            rw [hi]
      _ ≤ proj i t - proj i s :=
        sub_le_sub_left (hupper i s) _
      _ ≤ |proj i t - proj i s| :=
        le_abs_self _
      _ ≤ B * dist t s :=
        hproj i t s
  have hst' :
      h s - h t ≤ B * dist s t := by
    calc
      h s - h t =
          proj j s - h t := by
            rw [hj]
      _ ≤ proj j s - proj j t :=
        sub_le_sub_left (hupper j t) _
      _ ≤ |proj j s - proj j t| :=
        le_abs_self _
      _ ≤ B * dist s t :=
        hproj j s t
  have hst :
      h s - h t ≤ B * dist t s := by
    simpa [dist_comm] using hst'
  apply abs_le.mpr
  constructor
  · linarith
  · exact hts

/--
Weighted sum of three `B`-Lipschitz support terms.
-/
theorem weighted_support_lipschitz
    {Θ : Type*}
    [PseudoMetricSpace Θ]
    (h₁ h₂ h₃ : Θ → ℝ)
    (w₁ w₂ w₃ B : ℝ)
    (hw₁ : 0 ≤ w₁)
    (hw₂ : 0 ≤ w₂)
    (hw₃ : 0 ≤ w₃)
    (hlip₁ :
      ∀ t s,
        |h₁ t - h₁ s| ≤ B * dist t s)
    (hlip₂ :
      ∀ t s,
        |h₂ t - h₂ s| ≤ B * dist t s)
    (hlip₃ :
      ∀ t s,
        |h₃ t - h₃ s| ≤ B * dist t s)
    (t s : Θ) :
    |(w₁ * h₁ t + w₂ * h₂ t + w₃ * h₃ t) -
        (w₁ * h₁ s + w₂ * h₂ s + w₃ * h₃ s)|
      ≤
        (w₁ + w₂ + w₃) * B * dist t s := by
  let a : ℝ :=
    w₁ * (h₁ t - h₁ s)
  let b : ℝ :=
    w₂ * (h₂ t - h₂ s)
  let c : ℝ :=
    w₃ * (h₃ t - h₃ s)
  have hre :
      (w₁ * h₁ t + w₂ * h₂ t + w₃ * h₃ t) -
          (w₁ * h₁ s + w₂ * h₂ s + w₃ * h₃ s) =
        a + b + c := by
    dsimp [a, b, c]
    ring
  rw [hre]
  have hab :
      |a + b| ≤ |a| + |b| :=
    abs_add_le a b
  have habc :
      |a + b + c| ≤ |a + b| + |c| :=
    abs_add_le (a + b) c
  have htri :
      |a + b + c| ≤ |a| + |b| + |c| := by
    calc
      |a + b + c|
          ≤ |a + b| + |c| :=
        habc
      _ ≤ (|a| + |b|) + |c| :=
        add_le_add_left hab |c|
  have ha :
      |a| ≤ w₁ * (B * dist t s) := by
    dsimp [a]
    rw [abs_mul, abs_of_nonneg hw₁]
    exact
      mul_le_mul_of_nonneg_left
        (hlip₁ t s)
        hw₁
  have hb :
      |b| ≤ w₂ * (B * dist t s) := by
    dsimp [b]
    rw [abs_mul, abs_of_nonneg hw₂]
    exact
      mul_le_mul_of_nonneg_left
        (hlip₂ t s)
        hw₂
  have hc :
      |c| ≤ w₃ * (B * dist t s) := by
    dsimp [c]
    rw [abs_mul, abs_of_nonneg hw₃]
    exact
      mul_le_mul_of_nonneg_left
        (hlip₃ t s)
        hw₃
  calc
    |a + b + c|
        ≤ |a| + |b| + |c| :=
      htri
    _ ≤
        w₁ * (B * dist t s) +
          w₂ * (B * dist t s) +
          w₃ * (B * dist t s) := by
      linarith
    _ =
        (w₁ + w₂ + w₃) * B * dist t s := by
      ring

/--
Generic collocation estimate.

For a uniform `M`-point circular grid, instantiate
`δ = Real.pi / M`.
-/
theorem collocation_error
    {Θ ι : Type*}
    [PseudoMetricSpace Θ]
    (F : Θ → ℝ)
    (grid : ι → Θ)
    (c L δ : ℝ)
    (hL : 0 ≤ L)
    (hnear :
      ∀ t,
        ∃ j,
          dist t (grid j) ≤ δ)
    (hLip :
      ∀ t s,
        |F t - F s| ≤ L * dist t s)
    (hcoll :
      ∀ j,
        c ≤ F (grid j)) :
    ∀ t,
      c - L * δ ≤ F t := by
  intro t
  rcases hnear t with ⟨j, hj⟩
  have hdist :
      L * dist (grid j) t ≤ L * δ := by
    apply mul_le_mul_of_nonneg_left _ hL
    simpa [dist_comm] using hj
  have habs :
      |F (grid j) - F t| ≤ L * δ :=
    (hLip (grid j) t).trans hdist
  have hdiff :
      F (grid j) - F t ≤ L * δ :=
    (le_abs_self (F (grid j) - F t)).trans habs
  linarith [hcoll j]

/--
Dilation repairs a uniform deficit `δ < c`.
-/
theorem dilation_repairs_constraint
    {c δ F : ℝ}
    (hc : 0 < c)
    (hδ : δ < c)
    (hF : c - δ ≤ F) :
    c ≤ (c / (c - δ)) * F := by
  have hden :
      0 < c - δ :=
    sub_pos.mpr hδ
  have hfac :
      0 ≤ c / (c - δ) :=
    div_nonneg
      (le_of_lt hc)
      (le_of_lt hden)
  calc
    c =
        (c / (c - δ)) * (c - δ) := by
      rw [div_mul_cancel₀ c (ne_of_gt hden)]
    _ ≤
        (c / (c - δ)) * F :=
      mul_le_mul_of_nonneg_left hF hfac

/--
Fixed-complexity convergence by a quantitative squeeze estimate.
-/
theorem fixedK_value_convergence
    (vK : ℝ)
    (vKM : ℕ → ℝ)
    (a : ℝ)
    (hlower :
      ∀ᶠ M : ℕ in atTop,
        (1 - a / (M : ℝ)) * vK ≤ vKM M)
    (hupper :
      ∀ᶠ M : ℕ in atTop,
        vKM M ≤ vK) :
    Tendsto vKM atTop (nhds vK) := by
  have hzero :
      Tendsto
        (fun M : ℕ => a / (M : ℝ))
        atTop
        (nhds 0) :=
    tendsto_const_div_atTop_nhds_zero_nat a
  have hone :
      Tendsto
        (fun _ : ℕ => (1 : ℝ))
        atTop
        (nhds 1) :=
    tendsto_const_nhds
  have hsub :
      Tendsto
        (fun M : ℕ => 1 - a / (M : ℝ))
        atTop
        (nhds 1) := by
    simpa using hone.sub hzero
  have hvKconst :
      Tendsto
        (fun _ : ℕ => vK)
        atTop
        (nhds vK) :=
    tendsto_const_nhds
  have hlow :
      Tendsto
        (fun M : ℕ =>
          (1 - a / (M : ℝ)) * vK)
        atTop
        (nhds vK) := by
    simpa using hsub.mul hvKconst
  exact
    tendsto_of_tendsto_of_tendsto_of_le_of_le'
      hlow
      hvKconst
      hlower
      hupper

/--
Joint convergence.

`vK n` is the exact polygonal value, `vKM n` is the collocation
value, and `err n` is the multiplicative repair loss.
-/
theorem joint_value_convergence
    (v : ℝ)
    (vK vKM err : ℕ → ℝ)
    (hvK :
      Tendsto vK atTop (nhds v))
    (herr :
      Tendsto err atTop (nhds 0))
    (hlower :
      ∀ᶠ n : ℕ in atTop,
        (1 - err n) * v ≤ vKM n)
    (hupper :
      ∀ᶠ n : ℕ in atTop,
        vKM n ≤ vK n) :
    Tendsto vKM atTop (nhds v) := by
  have hone :
      Tendsto
        (fun _ : ℕ => (1 : ℝ))
        atTop
        (nhds 1) :=
    tendsto_const_nhds
  have hsub :
      Tendsto
        (fun n : ℕ => 1 - err n)
        atTop
        (nhds 1) := by
    simpa using hone.sub herr
  have hvconst :
      Tendsto
        (fun _ : ℕ => v)
        atTop
        (nhds v) :=
    tendsto_const_nhds
  have hlow :
      Tendsto
        (fun n : ℕ => (1 - err n) * v)
        atTop
        (nhds v) := by
    simpa using hsub.mul hvconst
  exact
    tendsto_of_tendsto_of_tendsto_of_le_of_le'
      hlow
      hvK
      hlower
      hupper

end TriangleCovering
\end{lstlisting}

\vspace{1em}
\textbf{AI usage disclosure:} The language of this paper was polished by GPT-5.6 sol and Gemini 3.1 pro. The formalized proofs in Lean 4 code was assisted by GPT-5.6 sol. The AI tools were not used to generate substantive content and construct logical arguments. All intellectual contributions, theoretical frameworks, and formulas are original work of the author.

~\\
College of Engineering and Computer Science, University of Central Florida, Orlando, FL, USA

Email: \underline{zhipeng.deng@ucf.edu}

\end{document}